%% file: ms.tex
\documentclass[dvipsnames,11pt,reqno]{amsart}
\usepackage{amssymb}
\usepackage{amsthm}
\usepackage{mathtools}
\usepackage[a4paper,left=3cm,right=3cm,top=3cm,bottom=3cm]{geometry}
\usepackage[T1]{fontenc}
\newcommand\numberthis{\addtocounter{equation}{1}\tag{\theequation}}
\newcounter{n}
\numberwithin{n}{section}
\theoremstyle{plain}
\newtheorem{lemma}[n]{Lemma}
\newtheorem{theorem}[n]{Theorem}
\newtheorem*{theorem*}{Theorem}
\newtheorem*{corollary*}{Corollary}

\newtheorem{proposition}[n]{Proposition}
\newtheorem{corollary}[n]{Corollary}
\theoremstyle{definition}

\newtheorem*{remark*}  {Remark}

\DeclareMathOperator*{\esssup}{ess\,sup}
\DeclareMathOperator*{\essinf}{ess\,inf}
\renewcommand{\setminus}{\smallsetminus}

\usepackage[shortlabels]{enumitem}
\usepackage{tikz,hyperref}
\definecolor{colorlinks}{RGB}{0, 24, 168}
\definecolor{colorcites}{RGB}{124, 10, 2}
\hypersetup{
    colorlinks=true,
    linkcolor=colorlinks,
    citecolor=colorcites,
    urlcolor=colorlinks,
    pdfborder={0 0 0}
}
\usepackage{setspace}
\usepackage{xcolor}
\title[Variational principle for weakly dependent random fields]
{Variational principle for\\weakly dependent random fields}
\author{Piet Lammers}
\address{Statistical Laboratory, Centre for Mathematical Sciences, University of Cambridge}
\email{p.g.lammers@statslab.cam.ac.uk}
\author{Martin Tassy}
\address{Department of Mathematics, Dartmouth College}
\email{mtassy@math.dartmouth.edu}
\begin{document}

\makeatletter
\@namedef{subjclassname@2020}{\textup{2020} Mathematics Subject Classification}
\makeatother

\subjclass[2020]{Primary 82B20, 82B44}
\keywords{Gibbs measures, entropy, free energy, random fields, variational principle}
\input{new/00_abstract.tex}
\maketitle

\setcounter{tocdepth}{1}
\tableofcontents

\input{new/01_intro.tex}
\input{new/02_prelim.tex}
\input{new/03_main.tex}
\input{new/04_SFE.tex}
\input{new/05_minimisers.tex}
\input{new/06_00_applications.tex}
\input{new/06_01_absolutely_summable.tex}
\input{new/06_02_RCM.tex}
\input{new/06_03_weight.tex}
\input{new/06_04_Griffiths.tex}

\input{new/99_acknowledgement.tex}

\bibliography{main.bib,clean.bib}
\bibliographystyle{amsalpha}

\end{document}

%% file: new/00_abstract.tex
\begin{abstract}
Using an alternative notion of entropy introduced by Datta, the max-entropy, we present a new simplified framework to study the minimizers of the specific free energy for random fields which are weakly dependent in the sense of Lewis, Pfister, and Sullivan. The framework is then applied to derive the variational principle for the loop $O(n)$ model and the Ising model in a random percolation environment in the nonmagnetic phase, and we explain how to extend the variational principle to similar models. To demonstrate the generality of the framework, we indicate how to naturally fit into it the variational principle for models with an absolutely summable interaction potential, and for the random-cluster model.
\end{abstract}

%% file: new/01_intro.tex

\section{Introduction}

\subsection{Random fields with long-range interactions}
One of the great results in statistical physics is the variational principle,
which asserts that a shift-invariant infinite-volume measure is a Gibbs measure
if and only if it minimizes the specific free energy.
The class of models which fall under the scope of the variational principle is extremely broad.
Models for which the interaction potential is absolutely summable were covered by Georgii~\cite{GEORGII}.
There have been numerous attempts to extend or generalize the variational principle beyond,
often in relation to a study of the points where \emph{continuity}
or \emph{quasilocality} of the specification
fails.
Such points are characterized by non-vanishing long-range interactions,
and appear naturally
in, for example, the random-cluster model~\cite{GRIMMETT,SEPPALAINEN},
the loop $O(n)$ model~\cite{PELED},
and several models in a random environment such
as the Ising model in a percolation environment~\cite{VanENTER}.
A non-exhaustive list of the study of the variational principle
for specifications which are not quasilocal includes~\cite{VanEnterFernandezSokal,PFISTER,SEPPALAINEN,MAES,VanENTER,VanEnterVerbitskiy,KULSKE}.
Further investigation into the variational principle
was carried out in relation to
renormalization \cite{MR1706769,MR1964281},
the large deviations principle \cite{MR1227034,MR1240718,MR1344727},
and projections or restrictions of Gibbs measures
\cite{MR1706773,MR2791060}.
Other works on the variational principle in the infinite-volume setting include \cite{Stroock,ZEGARLINSKI,FERNANDEZ}.
The variational principle is known to fail for some models,
most notably the random field Ising model~\cite{KULSKE}, which was known to exhibit phase transition~\cite{bricmont1988}.
Despite those efforts there are still some interesting models for which it is not known if the variational principle holds true or not.
Among those are various models of random fields in random environments:
a noteworthy example is the Ising model on a random subgraph of the square lattice obtained from independent percolation.
The inherent problem derives from the fact that the strength of the interactions between particles does not decay uniformly with the range.

This model belongs to a large, natural class of models known as
\emph{weakly dependent}: this term is due to Lewis, Pfister and Sullivan~\cite{LPS}.
We develop a streamlined framework for studying the minimizers of the specific free energy
within this class.
The framework allows one to efficiently deduce the variational principle for many
interesting
weakly dependent models.
Our discussion reviews the absolutely summable setting of~\cite{GEORGII},
and
the random-cluster model~\cite{SEPPALAINEN} (see~\cite{GRIMMETT} for a general introduction).
We break new ground by
 proving the variational principle
for the Ising model in a random environment, in the nonmagnetic phase.
This significantly extends the results of~\cite{KULSKE}.
We furthermore deduce the variational principle for the loop $O(n)$ model
(see~\cite{PELED} for a general introduction)
by extension of the discussion of the random-cluster model,
and we explain how these models represent any model
where the nonvanishing long-range interaction is due
to a potential associated with clusters, level sets, paths, or other large geometrical objects
that arise from the local structure.

\subsection{The specific free energy}
The specific free energy and
a suitable characterization for it
are of central importance
to the study of the variational principle.
A natural first question is thus to ask
about restrictions on the
model that guarantee that the specific free energy is well-behaved.
Candidates are the previously mentioned \emph{weakly dependent} \cite{LPS},
and the more general \emph{asymptotically decoupled}.
The latter was introduced by Pfister~\cite{Pfister2002}.
While either restriction guarantees a well-defined specific free energy,
the former is more amenable to arguments
involving regular conditional probability distributions,
and is therefore better for studying the variational principle.
Remark that we shall define
the specific free energy in terms of the specification
that characterizes the model,
unlike in \cite{LPS,Pfister2002} where
it is defined in terms of a reference random field.
Our definition of \emph{weakly dependent} is therefore cosmetically different.

There is a simple and natural definition
of
a weakly dependent specification once we introduce the max-entropy of two measures.
The max-entropy of some measure $\mu$ relative to another measure $\nu$ equals
\[
\mathcal H^\infty(\mu|\nu):=\log\inf\{\lambda\geq 0:\mu\leq\lambda\nu\},
\]
and was introduced by Datta in~\cite{DATTA}.
We call a specification \emph{weakly dependent} if the max-entropy between any two finite-volume Gibbs measures
on a box $\Lambda\subset\mathbb Z^d$ is of order $o(|\Lambda|)$ as $\Lambda$ grows large.

The class of weakly dependent models is rich, and it is not hard to prove that the various models that were mentioned are all weakly dependent.
If the model of interest is weakly dependent,
then the specific free energy has all the usual properties:
its level sets (which are sets of shift-invariant random fields) are compact in the topology of local convergence,
and there exist shift-invariant random fields that have zero specific free energy.

\subsection{Main results}
Consider a weakly dependent specification.
We call a random field
a \emph{minimizer} if it is shift-invariant
and has zero specific free energy with respect
to this specification.
It is a corollary of the definition
of the specific free energy that shift-invariant Dobrushin-Lanford-Ruelle (DLR) states are minimizers.
We show that
a shift-invariant random field is a minimizer
if and only if it is a limit of finite-volume Gibbs measures, where we allow mixed boundary conditions.
If $\mu$ is a minimizer,
then we derive properties of the conditional probability distribution of $\mu$ in a box $\Lambda$,
conditioned on what happens outside of $\Lambda$.
If $\mu$ is supported on the points of continuity of
the specification corresponding to the model,
then we show that $\mu$ is a DLR state, and almost Gibbs.
In general,
we demonstrate that all minimizers have finite energy in the sense of Burton and Keane, so that we are able to make their
case for almost sure uniqueness of the infinite cluster (if this is relevant for the model under consideration).

The variational principle asserts that the
minimizers of the specific free energy coincide
with the shift-invariant almost Gibbs measures.
The framework provides a clear route to demonstrating its validity for weakly dependent models:
it is sufficient to prove
that minimizers of the specific free energy are supported on the points of continuity of the specification,
and in deriving this one may assume all the properties that
minimizers of the specific free energy automatically have.

We apply the framework to all models that were previously mentioned.
First, we show how to fit into our framework the known variational principles
for models with an absolutely summable interaction potential \cite{GEORGII},
and for the random-cluster model \cite{SEPPALAINEN}.
Then, we derive the variational principle for the
loop $O(n)$ model, and by extension we assert that
the variational principle must hold true for a large
class of models where the long-range interaction is due to weight
on percolation clusters (such as for the random-cluster model),
level sets, loops, or other large geometrical objects which
arise from the local structure.
Next, we derive the variational principle for the
Ising model in a random percolation environment
in the nonmagnetic phase.
This improves upon the work of \cite{KULSKE},
where the same result is established for the phase where the random environment does not percolate.
The authors believe that for a large class of models in
a random environment,
the proposed framework significantly reduces the complexity
of determining wether or not the variational principle holds true.

Finally, it should be remarked that in all our work we shall never require the state space to be finite;
the framework works for any standard Borel space,
much like the setting of Georgii~\cite{GEORGII}.

\subsection{Structure} The article is organized as follows. In Section~\ref{s_definitions} we introduce the various mathematical objects necessary to define and study the specific free energy. In Section~\ref{s_main} we give a presentation of our main results. In Section~\ref{s_sfe} we show how to define the specific free energy for weakly dependent specifications, and we prove some of its properties. In Section~\ref{s_minimizers} we give a characterization of the minimizers of the specific free energy. In Section~\ref{s_applications} we show how to derive easily from our framework various versions of the variational principle.


%% file: new/02_prelim.tex

\section{Definitions}\label{s_definitions}

If $(X,\mathcal X)$ is any measurable space,
then write $\mathcal P(X,\mathcal X)$ for the set of probability measures on $(X,\mathcal X)$, and $\mathcal M(X,\mathcal X)$ for the set of
$\sigma$-finite
measures $\mu$ with $\mu(X)>0$.
In this paper we only consider measurable spaces that are standard
Borel spaces.
We shall follows the notation of Georgii~\cite{GEORGII} wherever possible.

\subsection{Random fields}

We are concerned with the study of random fields.
Fix a dimension $d\in\mathbb N$ and a standard Borel space $(E,\mathcal E)$ throughout this article.
The set $S:=\mathbb Z^d$ is called the \emph{parameter set},
and $(E,\mathcal E)$ is called the \emph{state space}.
Elements of $S$ are called \emph{sites}.
A \emph{configuration} is a function $\omega$ that assigns to each
site $x\in S$ a state $\omega_x\in E$.
Write $\Omega:=E^S$ for the set of configurations,
and $\mathcal F$ for the product $\sigma$-algebra
$\mathcal E^S$ on $\Omega$.
A \emph{random field} is a probability measure on configurations:
the set of random fields is $\mathcal P(\Omega,\mathcal F)$.

Define, for each site $x\in S$, the measurable function $\sigma_x:\Omega\to E,\,\omega\mapsto \omega_x$.
For any $\Lambda\subset S$, we shall write
$\mathcal F_\Lambda:=\sigma(\sigma_x:x\in\Lambda)\subset \mathcal F$.
Write furthermore $\sigma_\Lambda$ for the canonical projection map
$\Omega=E^S\to E^\Lambda$,
and observe that $\sigma_\Lambda$ extends canonically to a bijection
from $\mathcal F_\Lambda$ to $\mathcal E^\Lambda$
and to a bijection from $\mathcal P(\Omega,\mathcal F_\Lambda)$
to $\mathcal P(E^\Lambda,\mathcal E^\Lambda)$.
Define $\omega_\Lambda:=\sigma_\Lambda(\omega)$ for $\omega\in\Omega$,
and
if $\mu\in\mathcal P(\Omega,\mathcal X)$
for some $\mathcal F_\Lambda\subset\mathcal X\subset \mathcal F$,
then write $\mu_\Lambda:=\sigma_\Lambda(\mu)\in\mathcal P(E^\Lambda,\mathcal E^\Lambda)$.
If $f$ is an $\mathcal F_\Lambda$-measurable function on $\Omega$
and $g$ an $\mathcal E^\Lambda$-measurable function on $E^\Lambda$,
then we shall without further notice write $f$ for the $\mathcal E^\Lambda$-measurable function $f\circ \sigma_\Lambda^{-1}$ on $E^\Lambda$
and $g$ for the $\mathcal F_\Lambda$-measurable function $g\circ\sigma_\Lambda$ on $\Omega$.
Finally, if $\Lambda\subset\Delta\subset S$, then write
also $\sigma_\Lambda$ for the canonical projection map
$E^\Delta\to E^\Lambda$, and if $\omega\in E^\Lambda$
and $\zeta\in E^{\Delta\setminus\Lambda}$,
then write $\omega\zeta$ for the unique element of $E^\Delta$
such that $\sigma_\Lambda(\omega\zeta)=\omega$
and $\sigma_{\Delta\setminus\Lambda}(\omega\zeta)=\zeta$.

Define, for every $x\in\mathbb Z^d$, the map $\theta_x:\mathbb Z^d\to\mathbb Z^d,\,y\mapsto y+x$.
Each map $\theta_x$ is called a \emph{shift}.
Write $\Theta$ for the set of shifts, that is, $\Theta=\{\theta_x:x\in\mathbb Z^d\}$.
If $\omega\in\Omega$ and $\theta\in\Theta$, then write $\theta\omega$
for the configuration in $\Omega$
satisfying $(\theta\omega)_x=\omega_{\theta x}$ for every $x\in S$.
Similarly, define $\theta A:=\{\theta\omega:\omega\in A\}$
for $A\in\mathcal F$.
A random field $\mu\in\mathcal P(\Omega,\mathcal F)$
is called \emph{shift-invariant} if $\mu(\theta A)=\mu(A)$ for any $A\in\mathcal F$ and $\theta\in\Theta$.
Write $\mathcal P_\Theta(\Omega,\mathcal F)$
for the collection of shift-invariant random fields.

\subsection{Entropy and max-entropy}
Consider two $\sigma$-finite measures $\mu,\nu\in\mathcal M(X,\mathcal X)$
on a standard Borel space $(X,\mathcal X)$.
The \emph{entropy} of $\mu$ relative to $\nu$
is defined by
\[
  \mathcal H(\mu|\nu):=\begin{cases}
    \mu(\log f)=\nu(f\log f)&\text{if $\mu\ll\nu$ where $f:=d\mu/d\nu$,}\\
    \infty&\text{otherwise.}
\end{cases}
\]
The \emph{max-entropy} of $\mu$ relative to $\nu$
is defined by
\[
  \mathcal H^\infty(\mu|\nu):=
  \log\inf\{\lambda\geq 0:\mu\leq\lambda\nu\}
  =
  \begin{cases}
    \esssup\log f&\text{if $\mu\ll\nu$ where $f:=d\mu/d\nu$,}\\
    \infty&\text{otherwise.}
  \end{cases}
\]
Note that
both
entropies are nonnegative when $\mu$ and $\nu$
are probability measures --- if they are indeed probability measures,
then each entropy equals zero if and only if $\mu=\nu$.
If $\mathcal Y$ is a sub-$\sigma$-algebra
of $\mathcal X$,
then define $\mathcal H_\mathcal Y(\mu|\nu):=\mathcal H(\mu|_\mathcal Y|\nu|_\mathcal Y)$.
If $(X,\mathcal X)=(\Omega,\mathcal F)$
and $\Lambda\in\mathcal S$,
then abbreviate
$\mathcal H_{\mathcal F_\Lambda}(\mu|\nu)$
to $\mathcal H_\Lambda(\mu|\nu)$.
Introduce a similar definition for
$\mathcal H_\Lambda^\infty(\mu|\nu)$.
Finally, define the \emph{max-diameter} of
a nonempty set $\mathcal B\subset \mathcal M(X,\mathcal X)$
by
\[
  \operatorname{Diam}^\infty\mathcal B:=
  \sup_{\mu,\nu\in\mathcal B}
  \mathcal H^\infty(\mu|\nu)\geq 0,
\]
where we observe equality if and only if $\mathcal B$ contains exactly one measure.

For probability measures $\mu,\nu\in\mathcal P(X,\mathcal X)$,
we always have $\mathcal H(\mu|\nu)\leq\mathcal H^\infty(\mu|\nu)$.
It is possible however that $\mathcal H^\infty(\mu|\nu)$
is large and $\mathcal H(\mu|\nu)$ small,
for example if the Radon-Nikodym derivative $f:=d\mu/d\nu$
is large on a very small portion of $(X,\mathcal X)$.
The max-entropy should be understood as
a sort of $L^\infty$-version of the usual entropy.
The max-entropy and the max-diameter prove to be efficient
tools for selecting the class of models
for which the theory in this article works.
The usual entropy however, is sometimes easier
to work with due to a number of standard identities that are available; see for example (\ref{lemma_superadditivity_rcpd_decomposition}) in the proof of Lemma~\ref{lemma_superadditivity}.

\subsection{Weakly dependent specifications}
\label{subsec_weakly}
A \emph{specification} is a family $\gamma=(\gamma_\Lambda)_{\Lambda\in\mathcal S}$
with the following properties:
\begin{enumerate}
  \item Each member $\gamma_\Lambda$ is a probability kernel from $(\Omega,\mathcal F_{S\setminus\Lambda})$
  to $(\Omega,\mathcal F)$,
  \item Each member $\gamma_\Lambda$ satisfies $\gamma_\Lambda(A,\omega)=1(\omega\in A)$
  whenever $A\in \mathcal F_{S\setminus\Lambda}$,
  \item If $\Lambda\subset\Delta\in\mathcal S$,
  then $\gamma_\Delta=\gamma_\Delta\gamma_\Lambda$.
\end{enumerate}
A member $\gamma_\Lambda$ is called \emph{proper} if it has
the second property;
the family $\gamma$ is called \emph{consistent} if
it has the third property.
We fix a specification $\gamma$ throughout this article.
The specification $\gamma$ is called \emph{shift-invariant} if
$\gamma_{\theta\Lambda}(A,\omega)=\gamma_{\Lambda}(\theta A,\theta\omega)$
for any $\Lambda\in\mathcal S$, $A\in\mathcal F$, $\omega\in\Omega$,
 $\theta\in\Theta$.

Fix $\Lambda\in\mathcal S$,
and consider $\gamma_\Lambda$:
this is a probability kernel from $(\Omega,\mathcal F_{S\setminus\Lambda})$
to $(\Omega,\mathcal F)$.
Write $\hat\gamma_\Lambda$
for the unique probability kernel from $(\Omega,\mathcal F_{S\setminus\Lambda})$
to $(E^\Lambda,\mathcal E^\Lambda)$
such that $\hat\gamma_\Lambda(\cdot,\omega)=\sigma_\Lambda(\gamma_\Lambda(\cdot,\omega))$
for every $\omega\in\Omega$.
The measure $\hat\gamma_\Lambda(\cdot,\omega)$
is the \emph{finite-volume Gibbs measure} on $(E^\Lambda,\mathcal E^\Lambda)$
with \emph{deterministic boundary conditions} $\omega$.
Of course, the original kernel $\gamma_\Lambda$
can be recovered from $\hat\gamma_\Lambda$ through the equation
$\gamma_\Lambda(\cdot,\omega)=\hat\gamma_\Lambda(\cdot,\omega)\times\delta_{\omega_{S\setminus\Lambda}}$ --- this is because $\gamma_\Lambda$ is proper.
It is often more convenient to define $\hat\gamma_\Lambda$
than $\gamma_\Lambda$ when describing a specific model.

Now fix a random field $\mu\in\mathcal P(\Omega,\mathcal F)$,
and consider the finite-volume measure
$\mu\hat\gamma_\Lambda$.
This is the \emph{finite-volume Gibbs measure} on $(E^\Lambda,\mathcal E^\Lambda)$
with \emph{mixed boundary conditions} $\mu$.
Define
\[
  \mathcal B_\Lambda(\gamma):=\left\{
    \mu\hat\gamma_\Lambda:\mu\in\mathcal P(\Omega,\mathcal F)
  \right\}\subset\mathcal P(E^\Lambda,\mathcal E^\Lambda):
\]
the set of all such finite-volume Gibbs measures.
This set is convex because the set of all random fields is convex.
For each $n\in\mathbb N$, we use the notation $\Delta_n$
for the box
\[
  \Delta_n:=\{-n,\dots,n\}^d\in\mathcal S.
\]
The specification $\gamma$ is called \emph{weakly dependent} if
$\gamma$ is shift-invariant and satisfies
\[
  \operatorname{Diam}^\infty\mathcal B_{\Delta_n}(\gamma)=o(|\Delta_n|)
\]
as $n\to\infty$.
For technical reasons we also require that $\operatorname{Diam}^\infty\mathcal B_\Lambda(\gamma)$
is finite for any $\Lambda\in\mathcal S$; this additional condition is not restrictive.
Write
$\mathbb S$
for the collection of weakly dependent specifications.

Before proceeding, it is useful to remark that
\[
\operatorname{Diam}^\infty\mathcal B_{\Lambda}(\gamma)
:=
\sup_{\mu,\nu}\mathcal H^\infty(\mu\hat\gamma_\Lambda|\nu\hat\gamma_\Lambda)
=
\sup_{\omega,\zeta}\mathcal H^\infty(\hat\gamma_\Lambda(\cdot,\omega)|\hat\gamma_\Lambda(\cdot,\zeta));
\]
it is sufficient to consider deterministic boundary conditions in
calculating the max-diameter of $\mathcal B_\Lambda(\gamma)$.
This can be deduced from Fubini's theorem without effort.

\subsection{The specific free energy}
Consider a shift-invariant random field
$\mu$
and a weakly dependent specification $\gamma$.
The \emph{specific free energy (SFE)}
of $\mu$ relative
to $\gamma$ is defined by
\[
  h(\mu|\gamma):=\lim_{n\to\infty}|\Delta_n|^{-1}\mathcal H_{\Delta_n}(\mu|\nu\gamma_{\Delta_n})\in[0,\infty]
  \]
  where $\nu\in\mathcal P(\Omega,\mathcal F)$.
  Lemma~\ref{lemma_sfe} asserts
  that the limit exists for any $\nu$,
  and that this limit is independent of the choice of $\nu$.
  A shift-invariant random field $\mu$ with $h(\mu|\gamma)=0$ is called
  a \emph{minimizer} of $\gamma$.
  Write $h_0(\gamma)$ for the set of minimizers of $\gamma$.

Now take the perspective of a shift-invariant random field $\mu$.
The random field $\mu$ is called \emph{weakly dependent}
if $\mu\in h_0(\gamma)$
for some weakly dependent specification $\gamma$.
Write
$\mathbb F$
for the collection of weakly dependent random fields.
If $\mu$ is an arbitrary shift-invariant random field
and $\nu$ a weakly dependent random field,
then
 the \emph{specific free energy (SFE)}
of $\mu$ relative
to $\nu$ is defined by
\[
  h(\mu|\nu):=\lim_{n\to\infty}|\Delta_n|^{-1}\mathcal H_{\Delta_n}(\mu|\nu)\in[0,\infty].
  \]
  Lemma~\ref{lemma_duality_same_SFE} asserts that the limit converges for any choice of $\mu$
  and $\nu$.
  The quantity $h(\mu|\nu)$ is also sometimes called
  the \emph{entropy density} of $\mu$ with respect to $\nu$.
  Write $h_0(\nu)$ for the set of shift-invariant random fields
  $\mu$ with $h(\mu|\nu)=0$.
  Measures $\mu\in h_0(\nu)$
  are called \emph{minimizers} of $\nu$.

  \subsection{DLR states}
  Now consider a random field $\mu$
   and a finite set $\Lambda\in\mathcal S$.
  Write $\mu_\Lambda^\omega$
  for the regular conditional probability distribution (r.c.p.d.)\ on $(E^\Lambda,\mathcal E^\Lambda)$ of $\mu$
  corresponding to the projection map $\sigma_{S\setminus\Lambda}:\Omega\to E^{S\setminus\Lambda}$.
  Informally, this is the distribution of $\omega_\Lambda$ in $\mu$ given the states of $\omega$
  outside $\Lambda$.
  Suppose that we are given an arbitrary specification $\gamma$.
  A \emph{Dobrushin-Lanford-Ruelle (DLR) state}
  is a random field $\mu\in\mathcal P(\Omega,\mathcal F)$ which
   satisfies the \emph{DLR equation}
  $\mu=\mu\gamma_\Lambda$ for every $\Lambda\in\mathcal S$.
  In other words, $\mu$ is a DLR state if and only if
  $\mu_\Lambda^\omega=\hat\gamma(\cdot,\omega)$
  for $\mu$-almost every $\omega\in\Omega$,
  for each $\Lambda\in\mathcal S$.
  Write $\mathcal G(\gamma)$ for the set of DLR states,
  and $\mathcal G_\Theta(\gamma):=\mathcal G(\gamma)\cap\mathcal P_\Theta(\Omega,\mathcal F)$
  for the set of shift-invariant DLR states.

\subsection{Topologies}
The \emph{topology of local convergence} or \emph{$\mathcal L$-topology}
on $\Omega$
is the coarsest topology on $\Omega$
that makes the map $\omega\mapsto\omega_x$ continuous for every $x\in\mathbb Z^d$,
with respect to the discrete topology on $E$.
This means that $\omega^n\to\omega$
if and only if for any $\Lambda\in\mathcal S$, we have
$\omega^n_\Lambda=\omega_\Lambda$ for $n$ sufficiently large.

Consider an arbitrary standard Borel space $(X,\mathcal X)$.
The \emph{strong topology} on $\mathcal M(X,\mathcal X)$
is the coarsest topology that makes the map
$\mu\mapsto\mu(A)$ continuous for every $A\in\mathcal X$.
If $\mathcal B\subset\mathcal P(X,\mathcal X)$
is a convex set of probability measures subject to $\operatorname{Diam}^\infty\mathcal B$
being finite,
then write $\mathcal C(\mathcal B)$
for the closure of $\mathcal B$ in the strong topology.
In Lemma~\ref{lemma_mazur} we present an alternative definition
for $\mathcal C(\mathcal B)$, which we demonstrate is equivalent.

The \emph{topology of local convergence} or \emph{$\mathcal L$-topology}
on $\mathcal P(\Omega,\mathcal F)$
is the coarsest topology on $\mathcal P(\Omega,\mathcal F)$
that makes the map $\mu\mapsto\mu(A)$ continuous for every $A\in\cup_{\Lambda\in\mathcal S}\mathcal F_\Lambda$.
This means that $\mu^n\to\mu$
in the $\mathcal L$-topology if and only
if $\sigma_\Lambda(\mu^n)\to\mu_\Lambda$ in the strong topology
on $\mathcal P(E^\Lambda,\mathcal E^\Lambda)$
for every $\Lambda\in\mathcal S$.

Remark that we do not assume a topology on
the state space $E$.
A topology is not even implied, because the $\mathcal L$-topology
on measures originates from the strong topology on measures.
In some sense, the $\mathcal L$-topology thus
alludes to the discrete topology on $E$ --- this holds
true both when considered a topology on $\Omega$,
and when considered a topology on $\mathcal P(\Omega,\mathcal F)$.

\subsection{Limits of finite-volume Gibbs measures}
Let $\gamma$ be a weakly dependent specification.
Write $\mathcal W(\gamma)$ for the set of limits of
finite-volume Gibbs measures
in the $\mathcal L$-topology, that is,
\[
  \mathcal W(\gamma):=\left\{
    \mu\in\mathcal P(\Omega,\mathcal F):
    \text{$\nu^n\gamma_{\Delta_n}\to\mu$ in the $\mathcal L$-topology for some $(\nu^n)_{n\in\mathbb N}\subset \mathcal P(\Omega,\mathcal F)$}
  \right\}.
\]
It is immediate that $\mathcal G(\gamma)\subset\mathcal W(\gamma)$:
if $\mu\in\mathcal G(\gamma)$,
then $\mu\gamma_{\Delta_n}\to\mu$ and therefore $\mu\in\mathcal W(\gamma)$.
We write $\nu^n\gamma_{\Delta_n}$ in this definition
and not $\nu^n\hat\gamma_{\Delta_n}$
so that all measures live in the same space
and convergence in the $\mathcal L$-topology makes sense.
For simplicity the definition is in terms of the exhaustive sequence
$(\Delta_n)_{n\in\mathbb N}$;
it is straightforward to verify that the definition is
the same if we replace this sequence by any other increasing
exhaustive sequence.
Write $\mathcal W_\Theta(\gamma):=\mathcal W(\gamma)\cap\mathcal P_\Theta(\Omega,\mathcal F)$.
We shall later see that $h_0(\gamma)=\mathcal W_\Theta(\gamma)$.

\subsection{Continuity of the specification}
Consider a weakly dependent specification $\gamma$.
We are going to define more sets of finite-volume Gibbs measures,
now restricting the boundary conditions that are allowed.
For any $\Lambda,\Delta\in\mathcal S$ and $\omega\in\Omega$,
define
\[
  \mathcal B_{\Lambda,\Delta,\omega}(\gamma):=\{
  \mu\hat\gamma_\Lambda
  :
  \text{$\mu\in\mathcal P(\Omega,\mathcal F)$ such that $\mu_\Delta=\delta_{\omega_\Delta}$}
  \}
  \subset
  \mathcal B_\Lambda(\gamma)
  .
\]
The sets
$\mathcal B_\Lambda(\gamma)$ and $\mathcal B_{\Lambda,\Delta,\omega}(\gamma)$ are convex,
and
 $\mathcal B_{\Lambda,\Delta,\omega}(\gamma)$ is decreasing in $\Delta$.
Define
\[
\mathcal B_{\Lambda,\omega}(\gamma)
:=
\cap_{\Delta\in\mathcal S}
\mathcal C(\mathcal B_{\Lambda,\Delta,\omega}(\gamma))
=
\cap_{n\in\mathbb N}\mathcal C(\mathcal B_{\Lambda,\Delta_n,\omega}(\gamma)).
\]
Consider a measure $\mu\in\mathcal P(E^\Lambda,\mathcal E^\Lambda)$.
Then $\mu\in\mathcal B_{\Lambda,\omega}$
if and only if
$
  \nu^n\hat\gamma_\Lambda\to\mu
$
in the strong topology
for some sequence of random fields
$(\nu^n)_{n\in\mathbb N}$
converging to $\delta_\omega$ in the $\mathcal L$-topology.

The alternative characterization of $\mathcal B_{\Lambda,\omega}(\gamma)$ implies that
 $\delta_\omega\hat\gamma_\Lambda=\hat\gamma_\Lambda(\cdot,\omega)\in \mathcal B_{\Lambda,\omega}(\gamma)$.
 Define
 \begin{align*}
   \Omega_\gamma
   &:=
   \{\omega\in\Omega:\text{$\mathcal B_{\Lambda,\omega}(\gamma)=\{\hat\gamma_\Lambda(\cdot,\omega)\}$
   for any $\Lambda\in\mathcal S$}\}
   \\&\phantom{:}=
   \{\omega\in\Omega:\text{$\mathcal |\mathcal B_{\Lambda,\omega}(\gamma)|=1$
   for any $\Lambda\in\mathcal S$}\}
   .
 \end{align*}
In other words,
$\Omega_\gamma$
is the set of configurations $\omega\in\Omega$
such that the map $\zeta\mapsto\gamma_\Lambda(\cdot,\zeta)$
is continuous --- both sides endowed with the $\mathcal L$-topology --- at
$\omega$ for any $\Lambda\in\mathcal S$.
If $\omega\in\Omega_\gamma$,
then we say
 that the specification $\gamma$
is \emph{continuous} or \emph{quasilocal} at $\omega$.
If $\Omega_\gamma=\Omega$,
then each DLR state of $\gamma$ is also called a \emph{Gibbs measure}.
If $\mu\in\mathcal G(\gamma)$ and $\mu(\Omega_\gamma)=1$,
then $\mu$ is called an \emph{almost Gibbs measure}.
This makes sense even if $\Omega_\gamma\neq\Omega$.

%% file: new/03_main.tex

\section{Main results} \label{s_main}
\subsection{The specific free energy}
Consider a weakly dependent specification
$\gamma$.
We prove that for any shift-invariant random field $\mu$,
the SFE
\[
  h(\mu|\gamma)
  :=
  \lim_{n\to\infty}
  |\Delta_n|^{-1}\mathcal H_{\Delta_n}(\mu|\nu\gamma_{\Delta_n})\in[0,\infty]
\]
is well-defined, and independent of the choice of $\nu\in\mathcal P(\Omega,\mathcal F)$
(Lemma~\ref{lemma_sfe}).
Moreover,
we show that the level sets of the SFE --- given by $\{h(\cdot|\gamma)\leq C\}\subset\mathcal P_\Theta(\Omega,\mathcal F)$
for $C\in [0,\infty)$ --- are compact in the topology of local convergence,
and that $h_0(\gamma)=\{h(\cdot|\gamma)=0\}$
is nonempty (Lemma~\ref{lemma_level_set_compactness}).
We prove the first half of the variational principle,
which asserts that $\mathcal G_\Theta(\gamma)\subset h_0(\gamma)$ (Corollary~\ref{corollary_first_half}).

\subsection{Minimizers of the specific free energy}
\label{subsection_main_results_minimizers}
Next,
we focus on the set of minimizers $h_0(\gamma)$
of the weakly dependent specification $\gamma$.
We find some alternative characterizations for the set of minimizers.
In particular, if $\mu$ is a shift-invariant random field,
then the following are equivalent:
\begin{enumerate}
  \item $\mu\in h_0(\gamma)$, that is, $\mu$ is a minimizer of $\gamma$,
  \item $\mu\in \mathcal W(\gamma)$, that is, $\mu$ is a limit of finite-volume Gibbs measures,
  \item $\mu_{\Delta_n}\in \mathcal C(\mathcal B_{\Delta_n}(\gamma))$
  for each $n\in\mathbb N$;
\end{enumerate}
see
Lemma~\ref{lemma_finite_volume_limits} and
Corollary~\ref{corollary_replacing_symm_max_entr}.
Moreover, if $\mu$ is a minimizer,
then we demonstrate that
\begin{enumerate}
  \item $\mu$ is almost Gibbs if $\mu(\Omega_\gamma)=1$,
  \item $\mu_\Lambda^\omega\in\mathcal B_{\Lambda,\omega}$ for $\mu$-almost every $\omega$,
  for each $\Lambda\in\mathcal S$,
  \item $\mu$ has finite energy, in the sense of Burton and Keane.
\end{enumerate}
The first statement follows almost immediately from the second,
see Lemma~\ref{lemma_rcpd}
and Corollary~\ref{corollary_second_half}.
The third statement requires a short argument, see Corollary~\ref{lemma_applications_finite_energy}.

\subsection{The relation between $\mathbb F$ and $\mathbb S$}
\label{subsection_main_results_abstract_rela}
Now take a more abstract viewpoint,
and consider the set
of all weakly dependent random fields $\mathbb F$.
Choose a weakly dependent specification $\gamma\in\mathbb S$
and a minimizer $\nu\in\mathbb F$ of $\gamma$.
First, we prove that
$h(\mu|\nu)$ is well-defined and equal to $h(\mu|\gamma)$ for any shift-invariant random field $\mu$
 (Lemma~\ref{lemma_duality_same_SFE}).
This implies in particular that $h_0(\nu)=h_0(\gamma)$.
For $\mu,\nu\in\mathbb F$,
we declare $\mu\sim\nu$ if $h(\mu|\nu)=0$.
We prove that $\sim$ is an equivalence relation.
Write $\mathbb F^*$ for the
partition of $\mathbb F$ into equivalence classes.
This provides a canonical way to partition the set of
specifications $\mathbb S$ as well:
define the map
\[
\Xi:\mathbb S\to \mathbb F^*,\,\gamma\mapsto h_0(\gamma),
\]
and write $\mathbb S^*$ for the partition of $\mathbb S$ into
the level sets of $\Xi$.
This makes $\Xi$ into a bijection from $\mathbb S^*$ to $\mathbb F^*$ --- the
original map $\Xi$ was surjective by definition a weakly dependent random field.

\subsection{The variational principle in the weakly dependent setting}
Consider a weakly dependent specification $\gamma$.
The previous results provide efficient machinery for
attacking the variational principle.
Consider an arbitrary shift-invariant random field
$\mu$.
The variational principle asserts that
\[
  \numberthis
  \label{eq_var_pr_equivalence}
  \mu\in h_0(\gamma)\iff\text{$\mu$ is almost Gibbs with respect to $\gamma$}.
\]

To prove the variational principle for the model of interest,
we must always derive two results.
First, we must show that the specification $\gamma$ corresponding to the model is
indeed weakly dependent.
Second,
one must show that $\mu(\Omega_\gamma)=1$
for any minimizer $\mu$ of $\gamma$.
The variational principle then follows from Corollaries~\ref{corollary_first_half} and~\ref{corollary_second_half}.

Once weak dependence of the specification has been established,
the systematic study of the minimizers of the SFE
provides a number of useful properties
that minimizers of the SFE automatically have --- see Subsection~\ref{subsection_main_results_minimizers}.
This usually makes it easier to prove that $\mu(\Omega_\gamma)=1$
for arbitrary minimizers $\mu$.

We chose to formulate the variational principle
with respect to the standard entropy functional $\mathcal H$.
It is also possible to use the max-entropy $\mathcal H^\infty$
for this purpose.
To that end, simply replace
the set $h_0(\gamma)$ in (\ref{eq_var_pr_equivalence}) with the set
\[
h_0^\infty(\gamma):=\left\{
\mu\in\mathcal P_\Theta(\Omega,\mathcal F):
|\Delta_n|^{-1}
\mathcal H^\infty_{\Delta_n}(\mu|\nu\gamma_{\Delta_n})\to0
\right\}.
\]
We shall derive that $h_0^\infty(\gamma)=h_0(\gamma)$
for any weakly dependent specification $\gamma$.
The inclusion $h_0^\infty(\gamma)\subset h_0(\gamma)$
follows from the fact that $\mathcal H^\infty(\mu|\nu)\geq\mathcal H(\mu|\nu)$ for any $\mu,\nu\in\mathcal P(X,\mathcal X)$.
The other inclusion follows from
Lemma~\ref{lemma_mazur} and Corollary~\ref{corollary_replacing_symm_max_entr},
jointly with the definition of a weakly dependent specification.

\subsection{Applications}
The weakly dependent setting is very general:
it contains most nonpathological non-gradient models that do not have
some form of combinatorial exclusion (such as for example the dimer models,
which have a non-gradient interpretation but which are not weakly dependent).
We start by showing how to naturally fit two known variational principles into our framework.
Then we derive the variational principle for the loop $O(n)$ model,
and finally
 we derive new results for the Ising model in a random percolation environment.

In Subsection~\ref{subsection_abs_summable},
 we show how to
efficiently derive the variational principle
for models that are defined in terms of an absolutely summable interaction potential.
This setting is treated in the classical work of Georgii~\cite{GEORGII}.
For such models we find that $\Omega=\Omega_\gamma$,
meaning that all almost Gibbs measures are in fact Gibbs.
In Subsection~\ref{subsection_rcm},
 we show how
to derive the variational principle for the random-cluster model.
The original proof is due to Sepp\"al\"ainen~\cite{SEPPALAINEN}.
The proofs (the one of Sepp\"al\"ainen and the one presented here) rely on the finite energy of minimizers of the SFE,
which implies that there is at most one infinite cluster almost surely
with respect to such measures (see Burton and Keane~\cite{BURTON}).
In Subsection~\ref{subsection_loopOn},
we discuss how to derive the variational principle for the
loop $O(n)$ model, by analogy with the random-cluster model.
This result is new.
We also discuss how to derive the variational principle
for similar models.
In Subsection~\ref{subsection_ising}, we prove the variational principle
for the Ising model in a random percolation environment,
in the nonmagnetic phase.
Moreover, we demonstrate that the minimizer of the SFE is unique.
This is a new result. The variational principle
was previously derived for the subcritical percolation phase in~\cite{KULSKE}.

%% file: new/04_SFE.tex

\section{The specific free energy} \label{s_sfe}

This section has two main goals.
The first goal is to prove Lemma~\ref{lemma_sfe},
which asserts that the SFE is well-defined
for weakly dependent specifications.
It also provides some useful identities.
As an immediate corollary we observe that DLR states minimize the SFE.
The second goal is to prove Lemma~\ref{lemma_level_set_compactness},
which asserts that the level sets of the SFE are compact
in the $\mathcal L$-topology,
and that there exist measures with zero SFE.

\subsection{Consistency of the definition}
The definition of the SFE relies on two key lemmas.
Lemma~\ref{lemma_superadditivity} concerns superadditivity of a useful quantity.
Lemma~\ref{lemma_obvious_bound} bounds
the difference of two relative entropies in terms of the max-entropy.




\begin{lemma}
  \label{lemma_superadditivity}
  Let $\gamma$ denote any specification
  and $\mu$ a random field.
  Consider a finite pairwise disjoint family of finite sets $(\Lambda_k)_{1\leq k\leq n}\subset\mathcal S$, and write $\Lambda:=\cup_k\Lambda_k\in\mathcal S$.
  Then
  \[
  \inf_{\rho\in\mathcal P(\Omega,\mathcal F)}
    \mathcal H_\Lambda(\mu|\rho\gamma_\Lambda)
  \geq
  \sum_k
    \inf_{\rho\in\mathcal P(\Omega,\mathcal F)}
      \mathcal H_{\Lambda_k}(\mu|\rho\gamma_{\Lambda_k})
  .
  \]
\end{lemma}

\begin{proof}
  Fix $\nu\in\mathcal P(\Omega,\mathcal F)$,
  and replace $\nu$ by $\nu\gamma_\Lambda$ if the two are not equal.
  We must demonstrate that
  \[
    \mathcal H_\Lambda(\mu|\nu)
  \geq
  \sum_k
    \inf_{\rho\in\mathcal P(\Omega,\mathcal F)}
      \mathcal H_{\Lambda_k}(\mu|\rho\gamma_{\Lambda_k})
  .
  \]
  By induction, it is sufficient to consider the case $n=2$. We have
  \[
  \numberthis
  \label{lemma_superadditivity_rcpd_decomposition}
      \mathcal H_\Lambda(\mu|\nu)
    =
      \mathcal H_{\Lambda_1}(\mu|\nu)
    +
      \int_{E^{\Lambda_1}}
        \mathcal H_{\Lambda_2}(\mu^\zeta|\nu^\zeta)
      d\mu_{\Lambda_1}(\zeta),
  \]
  where $\mu^\zeta$ and $\nu^\zeta$ denote the r.c.p.d.\ on $(\Omega,\mathcal F)$ of
  $\mu$ and $\nu$ respectively corresponding to the projection map
  $\Omega\to E^{\Lambda_1}$.
  Recall that $\nu=\nu\gamma_{\Lambda}$.
  For the first term on the right in (\ref{lemma_superadditivity_rcpd_decomposition}), consistency of $\gamma$ implies that $\nu=\nu\gamma_{\Lambda_1}$ and
  \[
    \mathcal H_{\Lambda_1}(\mu|\nu)
    =\mathcal H_{\Lambda_1}(\mu|\nu\gamma_{\Lambda_1})
    \geq
    \inf_{\rho\in\mathcal P(\Omega,\mathcal F)}
      \mathcal H_{\Lambda_1}(\mu|\rho\gamma_{\Lambda_1}).
  \]
  The goal is to obtain a similar lower bound for the integral
  in (\ref{lemma_superadditivity_rcpd_decomposition}).
  Assume in the sequel that
  $\mathcal H_{\Lambda_1}(\mu|\nu)$ is finite;
  the lemma follows from (\ref{lemma_superadditivity_rcpd_decomposition}) if it is not.
  This means in particular that $\mu_{\Lambda_1}\ll\nu_{\Lambda_1}$.
  Formally, $\mu^\zeta$ and $\nu^\zeta$ are probability kernels
  from
  $(E^{\Lambda_1},\mathcal E^{\Lambda_1})$
  to $(\Omega,\mathcal F)$,
  which may be measured by $\mu_{\Lambda_1}$.
  Moreover, these kernels satisfy $\sigma_{\Lambda_1}(\mu^\zeta)=
  \sigma_{\Lambda_1}(\nu^\zeta)=\delta_\zeta$.
  First we assert that
  \[
  \int_{E^{\Lambda_1}}
    \mathcal H_{\Lambda_2}(\mu^\zeta|\nu^\zeta)
  d\mu_{\Lambda_1}(\zeta)
  =
    \mathcal H_{\Lambda}(\mu_{\Lambda_1}\mu^\zeta|\mu_{\Lambda_1}\nu^\zeta).
  \]
  It is straightforward to see
  that this holds true:
  an expansion of the expression on the right in this display
  similar to the expansion in (\ref{lemma_superadditivity_rcpd_decomposition})
  yields the integral on the left plus the entropy term $\mathcal H(\mu_{\Lambda_1},\mu_{\Lambda_1})=0$.
  It is clear that $\mu_{\Lambda_1}\mu^\zeta=\mu$.
  For the other kernel, we observe that
  $\nu^\zeta=\nu^\zeta\gamma_{\Lambda_2}$ by consistency
  for $\nu_{\Lambda_1}$-almost every $\zeta$,
  and therefore also for $\mu_{\Lambda_1}$-almost every $\zeta$.
  In particular, this means that
  \[
      \mathcal H_{\Lambda}(\mu_{\Lambda_1}\mu^\zeta|\mu_{\Lambda_1}\nu^\zeta)
      =
      \mathcal H_{\Lambda}(\mu|\mu_{\Lambda_1}\nu^\zeta\gamma_{\Lambda_2})
      \geq
      \mathcal H_{\Lambda_2}(\mu|\mu_{\Lambda_1}\nu^\zeta\gamma_{\Lambda_2})
      \geq
      \inf_{\rho\in\mathcal P(\Omega,\mathcal F)}
        \mathcal H_{\Lambda_2}(\mu|\rho\gamma_{\Lambda_2}).
        \qedhere
  \]
\end{proof}




\begin{lemma}
  \label{lemma_obvious_bound}
Let $(X,\mathcal X)$ denote a measurable space,
and consider $\mathcal B\subset \mathcal M(X,\mathcal X)$
with $\operatorname{Diam}^\infty\mathcal B$ finite.
Then for any finite measure $\mu\in\mathcal M(X,\mathcal X)$
and for any $\nu,\nu'\in\mathcal B$,
we have
\begin{equation*}
  |\mathcal H(\mu|\nu)-\mathcal H(\mu|\nu')|\leq
  \mu(X)\operatorname{Diam}^\infty\mathcal B,
\end{equation*}
where we interpret $|\infty-\infty|$ as $0$.
\end{lemma}
\begin{proof}
  Note that $\mu\ll\nu$ if and only if $\mu\ll\nu'$.
  Write $f:=d\mu/d\nu$ and $f':=d\mu/d\nu'$.
  Then $\mu$-almost everywhere $d\nu/d\nu'=f'/f$
  and  $|\log f'-\log f|\leq \operatorname{Diam}^\infty\mathcal B$.
  In particular,
  \begin{equation*}
    |\mathcal H(\mu|\nu)-\mathcal H(\mu|\nu')|
    =|\mu(\log f)-\mu(\log f')|
    \leq\mu(|\log f-\log f'|)
    \leq
    \mu(X)\operatorname{Diam}^\infty\mathcal B.\qedhere
  \end{equation*}
\end{proof}




\begin{lemma}
  \label{lemma_sfe}
  The specific free energy functional
  $h(\cdot|\gamma):\mathcal P_\Theta(\Omega,\mathcal F)\to[0,\infty]$ satisfies
  \begin{alignat}{2}
    \label{lemma_sfe_ori_lim}
    h(\mu|\gamma)
      &:=\lim\nolimits_{n\to\infty}&&|\Delta_n|^{-1}\mathcal H_{\Delta_n}(\mu|\nu\gamma_{\Delta_n})
  \\
  \label{lemma_sfe_ori_sup}
      &\phantom{:}=\sup\nolimits_{n\in\mathbb N}&&|\Delta_n|^{-1}(\mathcal H_{\Delta_n}(\mu|\nu\gamma_{\Delta_n})-\operatorname{Diam}^\infty\mathcal B_{\Delta_n}(\gamma))
  \\
    \label{lemma_sfe_inf_lim}  &\phantom{:}=\lim\nolimits_{n\to\infty}
        &&|\Delta_n|^{-1}
        \inf\nolimits_{\rho\in\mathcal P(\Omega,\mathcal F)}
          \mathcal H_{\Delta_n}(\mu|\rho\gamma_{\Delta_n})
    \\
    \label{lemma_sfe_inf_sup}
        & \phantom{:} =\sup\nolimits_{n\in\mathbb N}
            &&|\Delta_n|^{-1}
            \inf\nolimits_{\rho\in\mathcal P(\Omega,\mathcal F)}
              \mathcal H_{\Delta_n}(\mu|\rho\gamma_{\Delta_n})
  \end{alignat}
  for any weakly dependent specification $\gamma$ and for any $\nu\in\mathcal P(\Omega,\mathcal F)$.
\end{lemma}

\begin{proof}
  Together, Lemma~\ref{lemma_superadditivity} of the current paper and Lemma 15.11 of~\cite{GEORGII}
  assert that the sequence in (\ref{lemma_sfe_inf_lim}) converges, with limit (\ref{lemma_sfe_inf_sup}).
  Lemma~\ref{lemma_obvious_bound} and
  weak dependence of $\gamma$ imply that for any $\nu\in\mathcal P(\Omega,\mathcal F)$,
  \[
    \left|
      \mathcal H_{\Delta_n}(\mu|\nu\gamma_{\Delta_n})
    -
    \inf\nolimits_{\rho\in\mathcal P(\Omega,\mathcal F)}
      \mathcal H_{\Delta_n}(\mu|\rho\gamma_{\Delta_n})
    \right|
    \leq
    \operatorname{Diam}^\infty\mathcal B_{\Delta_n}(\gamma)=o(|\Delta_n|)
  \]
  as $n\to\infty$.
  This means that (\ref{lemma_sfe_ori_lim}) and (\ref{lemma_sfe_inf_lim}) are the same.
  The inequality in the display implies that each term in the supremum in (\ref{lemma_sfe_ori_sup})
  is bounded from above by the corresponding term in (\ref{lemma_sfe_inf_sup}),
  and therefore the supremum in (\ref{lemma_sfe_ori_sup}) is bounded from above by the supremum in
  (\ref{lemma_sfe_inf_sup}).
  However, the asymptotic bound on $\operatorname{Diam}^\infty\mathcal B_{\Delta_n}(\gamma)$ implies
  that the supremum in (\ref{lemma_sfe_ori_sup}) equals at least the limit in (\ref{lemma_sfe_ori_lim}).
  Conclude that (\ref{lemma_sfe_ori_lim}), (\ref{lemma_sfe_ori_sup}), (\ref{lemma_sfe_inf_lim}) and (\ref{lemma_sfe_inf_sup}) are all equal.
\end{proof}

\begin{corollary}
  \label{corollary_first_half}
  We have $\mathcal G_\Theta(\gamma)\subset h_0(\gamma)$ whenever $\gamma$ is weakly dependent.
\end{corollary}

\begin{proof}
  Consider $\mu\in\mathcal G_\Theta(\gamma)$, and apply the previous lemma with $\nu=\mu$.
\end{proof}

\subsection{Minimizers and level sets}




\begin{lemma}
  \label{lemma_level_set_compactness}
  Let $\gamma$ denote a weakly dependent specification.
  Then $\{h(\cdot|\gamma)\leq C\}$ is nonempty and compact in the $\mathcal L$-topology for any $C\in[0,\infty)$. In particular, $h_0(\gamma)$ is nonempty and compact in the $\mathcal L$-topology.
\end{lemma}

\begin{proof}
  The motivation for this lemma is standard; we include a proof for completeness.
  Fix a measure $\nu\in\mathcal P_\Theta(\Omega,\mathcal F)$ and a constant $C\in [0,\infty)$.
  Level sets of relative entropy are compact: in our setting
  \[
    \mathcal P_{n,C}:=\{
       \mu\in\mathcal P(E^{\Delta_n},\mathcal E^{\Delta_n}):
       \mathcal H(\mu|\nu\hat\gamma_{\Delta_n})\leq
       |\Delta_n|C+\operatorname{Diam}^\infty\mathcal B_{\Delta_n}(\gamma)
    \}
  \]
  is compact in the strong topology on $\mathcal P(E^{\Delta_n},\mathcal E^{\Delta_n})$ for any $n\in\mathbb N$.
  Equation~\ref{lemma_sfe_ori_sup} of Lemma~\ref{lemma_sfe}
  says that
  \[
  \{h(\cdot|\gamma)\leq C\}
  =
  \cap_{n\in\mathbb N}
  \{\mu\in\mathcal P_\Theta(\Omega,\mathcal F):\mu_{\Delta_n}\in\mathcal P_{n,C}\}.
  \]
  Let $(\mu^m)_{m\in\mathbb N}\subset\mathcal P(\Omega,\mathcal F)$
  denote a sequence of random fields --- not necessarily
  shift-invariant --- such that for any fixed $n\in\mathbb N$,
  we have
  $
    \sigma_{\Delta_n}(\mu^m)\in\mathcal P_{n,C}
  $
  for $m$ sufficiently large.
  By compactness of each set $\mathcal P_{n,C}$,
  a standard diagonalisation argument,
  and the Kolmogorov extension theorem,
  we obtain a subsequential limit $\mu\in\mathcal P(\Omega,\mathcal F)$
  of $(\mu^m)_{m\in\mathbb N}$ in the $\mathcal L$-topology
   with the property that $\mu_{\Delta_n}\in\mathcal P_{n,C}$
   for each $n\in\mathbb N$.

   For the lemma,
  it suffices to prove that $\{h(\cdot|\gamma)\leq C\}$ is compact and
  that $h_0(\gamma)$ is nonempty.
  Start with the former.
  Suppose that $(\mu^m)_{m\in\mathbb N}\subset \{h(\cdot|\gamma)\leq C\}$.
  Then $\sigma_{\Delta_n}(\mu^m)\in\mathcal P_{n,C}$
  for any $n,m\in\mathbb N$.
  Apply the previous argument to obtain a subsequential limit $\mu\in\mathcal P(\Omega,\mathcal F)$.
  Then $\mu$ must be shift-invariant because each $\mu^m$ is shift-invariant.
  The argument says moreover that
  $\mu_{\Delta_n}\in\mathcal P_{n,C}$
  for each $n\in\mathbb N$,
  that is,
  $h(\mu|\gamma)\leq C$.
  This proves that the level set $\{h(\cdot|\gamma)\leq C\}$ is compact.
  Finally, we prove that $h_0(\gamma)$ is nonempty.
  Set $C$ to $0$, and
  define
  \[
    \mu^m:=\frac{1}{|\Delta_m|}\sum_{x\in \Delta_m} \nu\gamma_{\Delta_{2m}+x}
    =
    \frac{1}{|\Delta_m|}\sum_{x\in \Delta_m}\theta_{x}\nu\gamma_{\Delta_{2m}}
    .
  \]
  The two measures are equal because $\nu$ is shift-invariant,
  and it is clear that any subsequential limit of $(\mu^m)_{m\in\mathbb N}$
  is also shift-invariant.
  Moreover, $\mu^m\gamma_{\Delta_n}=\mu^m$ whenever $m\geq n$
 because $\Delta_n\subset\Delta_m\subset \Delta_{2m}+x$
  for any $x\in\Delta_m$.
  This means that
  $\sigma_{\Delta_n}(\mu^m)\in \mathcal P_{n,0}$
   for $m$ sufficiently large, for each fixed $n\in\mathbb N$.
   The sequence thus has a subsequential limit $\mu$ in the $\mathcal L$-topology.
   This limit $\mu$ must satisfy $\mu_{\Delta_n}\in\mathcal P_{n,0}$
   for any $n$.
   Conclude that $\mu\in h_0(\gamma)$, that is, $h_0(\gamma)$ is nonempty.
\end{proof}

%% file: new/05_minimisers.tex

\section{Minimizers of the specific free energy} \label{s_minimizers}

\subsection{Mazur's lemma}

\begin{lemma}
  \label{lemma_mazur}
  Let $(X,\mathcal X)$ denote a standard Borel space and
  $\mathcal B$ a convex subset of $\mathcal P(X,\mathcal X)$
  subject to $\operatorname{Diam}^\infty\mathcal B$ being finite.
  Then the set
  \[
  \mathcal C:=
  \mathcal C(\mathcal B):=
  \left\{\mu\in\mathcal P(X,\mathcal X):\inf\nolimits_{\nu\in\mathcal B}\mathcal H(\mu|\nu)=0\right\}
  \]
  is compact in the strong topology on $\mathcal P(X,\mathcal X)$, satisfies $\operatorname{Diam}^\infty\mathcal C=\operatorname{Diam}^\infty\mathcal B$,
  and equals
  \begin{enumerate}
    \item The closure of $\mathcal B$ in the total variation topology,
    \item The closure of $\mathcal B$ in the strong topology.
  \end{enumerate}
\end{lemma}

This lemma is close to trivial when $E$ is finite, which is the case for many, but certainly not all, interesting models.
It is this lemma that makes the theory work also for models where
$(E,\mathcal E)$ is a general standard Borel space.
The lemma comes down to a straightforward application of
Mazur's lemma: a well-known result from functional analysis.
Remark that there is a confusing difference in naming
conventions for topologies between functional analysis
and measure theory, when the topologies
are on sets of (probability) measures.

\begin{proof}[Proof of Lemma~\ref{lemma_mazur}]
  Fix a measure $\lambda\in\mathcal B$; this measure will serve as reference measure.
  Write $f_\mu:=d\mu/d\lambda$
  for any $\sigma$-finite measure $\mu$ on $(X,\mathcal X)$
  that is absolutely continuous with respect
  to $\lambda$.
  For example, if $\mu\in\mathcal B$,
  then $\lambda$-almost everywhere $|\log f_\mu|\leq\operatorname{Diam}^\infty\mathcal B$.
  In particular, the map $\mu\mapsto f_\mu$
  injects $\mathcal B$ into $L^1(\lambda)$ ---
  the image of $\mathcal B$ under this map is also convex.
  Write $f^-$ for the lattice infimum of the family $\{f_\mu:\mu\in\mathcal B\}$; this is the largest $\mathcal X$-measurable function such that $\lambda$-almost everywhere $f^-\leq f_\mu$
   for each $\mu\in\mathcal B$.
   See Lemma 2.6 in~\cite{HAJLASZ}
   for existence and uniqueness of $f^-\in L^1(\lambda)$.
   Similarly, write $f^+$ for the lattice supremum of $\{f_\mu:\mu\in\mathcal B\}$.
   Observe that $\lambda$-almost everywhere
   $f^-\leq 1\leq f^+$ and
   \[
    0\leq \log\frac{f^+}{f^-}\leq \operatorname{Diam}^\infty\mathcal B;
   \]
   the former because $\lambda\in\mathcal B$, the latter
   follows from the definition of the diameter.
   In particular,
   \[e^{-\operatorname{Diam}^\infty\mathcal B}\leq\essinf\nolimits_\lambda f^\pm\leq\esssup\nolimits_\lambda f^\pm \leq e^{\operatorname{Diam}^\infty\mathcal B}.\]
   Define the measures $\lambda^\pm:=f^\pm\lambda$ --- these should be considered the lattice infimum and supremum of the set $\mathcal B$,
   and are independent of the choice of reference measure $\lambda\in\mathcal B$.
   A measure $\mu\in\mathcal P(X,\mathcal X)$
   must satisfy $\lambda^-\leq \mu\leq \lambda^+$
   if either $\mu\in\mathcal C$,
   or if $\mu$ is in the closure of $\mathcal B$ in the total variation topology,
   or if $\mu$ is in the closure of $\mathcal B$ in the strong topology.
   This also implies that $\lambda$-almost everywhere $f^-\leq f_\mu\leq f^+$.

   We first show that $\operatorname{Diam}^\infty\mathcal C=\operatorname{Diam}^\infty\mathcal B$.
   The previous observation implies that
   \[
    \operatorname{Diam}^\infty\mathcal C\leq \operatorname{Diam}^\infty\{\mu\in\mathcal P(X,\mathcal X):\lambda^-\leq \mu\leq \lambda^+\}=\mathcal H^\infty(\lambda^+|\lambda^-)=\operatorname{Diam}^\infty\mathcal B.
   \]
   Now $\mathcal B\subset\mathcal C$ and therefore $\operatorname{Diam}^\infty\mathcal B\leq\operatorname{Diam}^\infty\mathcal C$: we conclude that $\operatorname{Diam}^\infty\mathcal C=\operatorname{Diam}^\infty\mathcal B$.

   Fix a probability measure $\mu$ subject to
   $\lambda^-\leq \mu\leq \lambda^+$;
   the goal is to show that $\mu\in\mathcal C$
   if and only if $\mu$ is contained in the closure of $\mathcal B$ in the total variation topology.
   Fix a sequence $(\nu_n)_{n\in\mathbb N}\subset \mathcal B$.
   Observe that $d\mu/d\nu_n=f_\mu/f_{\nu_n}$,
   and that
   \[\textstyle
    \mathcal H(\mu|\nu_n)=\nu_n\left(
      \frac{f_\mu}{f_{\nu_n}}\log \frac{f_\mu}{f_{\nu_n}}
    \right)
    =\nu_n\left(\Xi\left(\frac{f_\mu}{f_{\nu_n}}\right)\right),
   \]
   where $\Xi:(0,\infty)\to[0,\infty)$
   is defined by $\Xi(x):=1-x+x\log x$.
   The function $\Xi$ is convex and attains its minimum
   $0$ at $x=1$ only.
   We observe that, as $n\to\infty$,
   \begin{align*}
     \mathcal H(\mu|\nu_n)\to 0 &\iff \nu_n(\Xi(f_\mu/f_{\nu_n}))\to 0\\
     &\numberthis\label{lemma_mazur_align_second}\iff \lambda(\Xi(f_\mu/f_{\nu_n}))\to 0\\
     &\numberthis\label{lemma_mazur_align_third}\iff f_{\nu_n}\to f_\mu\qquad\text{in $L^1(\lambda)$}\\
     &\numberthis\label{lemma_mazur_align_fourth}\iff \nu_n\to\mu\qquad \text{in total variation.}
   \end{align*}
   The equivalence in (\ref{lemma_mazur_align_second}) is due to the fact that
   $e^{-\operatorname{Diam}^\infty\mathcal B}\lambda\leq \nu_n\leq e^{\operatorname{Diam}^\infty\mathcal B}\lambda$
   for each $n\in\mathbb N$,
   and nonnegativity of $\Xi$.
   Equivalence in (\ref{lemma_mazur_align_third}) is due to said properties of the function $\Xi$,
   and the fact that all functions $f_\mu$ and $f_{\nu_n}$ are uniformly bounded away from zero and
   infinity.
   Equivalence in (\ref{lemma_mazur_align_fourth})
   is straightforward as
   $\lambda(|f_{\nu_n}-f_\mu|)$ equals the total variation distance from $\nu_n$ to $\mu$.
   We have now proven that $\mathcal C$ equals the closure of $\mathcal B$ in the total variation topology.


   Claim that the closure of $\mathcal B$ in the total variation topology equals
   the closure of $\mathcal B$ in the strong topology.
   The map $\mu\mapsto f_\mu$ is a bijection from
   the closure of $\mathcal B$ in the total variation topology to
   the closure of $\{f_\mu:\mu\in\mathcal B\}$ in the norm topology on $L^1(\lambda)$.
   The map $\mu\mapsto f_\mu$ is also a bijection from the closure of $\mathcal B$ in the strong topology
   to the closure of $\{f_\mu:\mu\in\mathcal B\}$ in the weak topology
   on $L^1(\lambda)$.
   The set $\{f_\mu:\mu\in\mathcal B\}$ is convex, and therefore
   Mazur's lemma asserts that the closure of $\{f_\mu:\mu\in\mathcal B\}$ in $L^1(\lambda)$
   is the same for the norm topology and for the weak topology.

   The set $\mathcal C$ is compact in the strong topology
   because it is closed in the strong topology
   and has finite max-diameter:
   it is a subset of the compact set $\{\mu\in\mathcal P(X,\mathcal X):\mathcal H(\mu|\lambda)\leq \operatorname{Diam}^\infty\mathcal C\}$.
\end{proof}

\begin{corollary}
  \label{corollary_replacing_symm_max_entr}
  Consider a weakly dependent specification $\gamma$,
  and a shift-invariant random field $\mu$.
  Then $\mu\in h_0(\gamma)$
  if and only if $\mu_{\Delta_n}\in\mathcal C(\mathcal B_{\Delta_n}(\gamma))$
  for each $n\in\mathbb N$.
\end{corollary}

\begin{proof}
  This is due to (\ref{lemma_sfe_inf_sup}) of Lemma~\ref{lemma_sfe}
  in combination with Lemma~\ref{lemma_mazur}.
\end{proof}

\subsection{Limits of finite-volume Gibbs measures}

\begin{lemma}
  \label{lemma_finite_volume_limits}
  If $\gamma$ is a weakly dependent specification,
  then $h_0(\gamma)=\mathcal W_\Theta(\gamma)$.
\end{lemma}

\begin{proof}
  If $\mu\in\mathcal W_\Theta(\gamma)$,
  then $\mu_{\Delta_n}\in\mathcal C(\mathcal B_{\Delta_n}(\gamma))$
  by definition of $\mathcal W(\gamma)$,
  and therefore $\mu\in h_0(\gamma)$
  by Corollary~\ref{corollary_replacing_symm_max_entr}.
  Now consider $\mu\in h_0(\gamma)$.
  For the lemma, it suffices to prove that
  $\mu\in\mathcal W_\Theta(\gamma)$.
  Again, Corollary~\ref{corollary_replacing_symm_max_entr}
  says that
  $
  \mu_{\Delta_n}\in\mathcal C(\mathcal B_{\Delta_n}(\gamma))
  $
  for each $n\in\mathbb N$.
   Write $d(\cdot,\cdot)$ for total variation distance.
   Lemma~\ref{lemma_mazur} implies that
   there exists a sequence of measures
   $(\nu^n)_{n\in\mathbb N}\subset\mathcal P(\Omega,\mathcal F)$
   such that
   \[d(\mu_{\Delta_n},\nu^n\hat\gamma_{\Delta_n})\leq 1/n\]
   for each $n\in\mathbb N$.
   Now for any $m\geq n$, we observe that
   \[
    d(\mu_{\Delta_n},\sigma_{\Delta_n}(\nu^m\gamma_{\Delta_m}))
    \leq
    d(\mu_{\Delta_m},\sigma_{\Delta_m}(\nu^m\gamma_{\Delta_m}))
    =
    d(\mu_{\Delta_m},\nu^m\hat\gamma_{\Delta_m})
    \leq1/m.
   \]
   In particular, $\sigma_{\Delta_n}(\nu^m\gamma_{\Delta_m})$
   approaches $\mu_{\Delta_n}$ in the total variation
   topology as $m\to\infty$, and therefore also in the strong topology.
   Conclude that $\nu^m\gamma_{\Delta_m}\to\mu$ in
   the topology of local convergence as $m\to\infty$.
   In other words, $\mu\in\mathcal W_\Theta(\gamma)$.
\end{proof}

\subsection{Regular conditional probability distributions}

Recall that $\mu_\Lambda^\omega$
denotes the r.c.p.d.\ on $(E^\Lambda,\mathcal E^\Lambda)$ of $\mu$
corresponding to the projection map $\sigma_{S\setminus\Lambda}:\Omega\to E^{S\setminus\Lambda}$,
where $\mu\in\mathcal P(\Omega,\mathcal F)$ is an arbitrary random field,
and $\Lambda\in\mathcal S$.
Recall also that we use the notation $\mathcal B_{\Lambda,\omega}(\gamma)$
for the set
\[
\mathcal B_{\Lambda,\omega}(\gamma)
:=
\cap_{\Delta\in\mathcal S}\mathcal C(\mathcal B_{\Lambda,\Delta,\omega}(\gamma))
=
\cap_{n\in\mathbb N}
\mathcal C(\mathcal B_{\Lambda,\Delta_n,\omega}(\gamma)).
\]

\begin{lemma}
  \label{lemma_rcpd}
  Let $\gamma$ be a weakly dependent specification,
  and fix a minimizer  $\mu\in h_0(\gamma)$
  and a finite set $\Lambda\in\mathcal S$.
  Then
  the r.c.p.d.\ of $\mu$ satisfies
  $\mu_\Lambda^\omega\in\mathcal B_{\Lambda,\omega}(\gamma)$
  for $\mu$-almost every $\omega$.
\end{lemma}

\begin{proof}
  Fix an arbitrary set $\Delta\in\mathcal S$ that contains $\Lambda$.
  For the lemma it suffices to show that
  $
  \mu_\Lambda^\omega\in\mathcal C(\mathcal B_{\Lambda,\Delta,\omega}(\gamma))
  $
  for $\mu$-almost every $\omega$.
  Write $\mu_n^\omega$
  for the r.c.p.d.\ of $\mu$ on $(E^\Lambda,\mathcal E^\Lambda)$
  corresponding to the natural projection map
  $\sigma_{\Delta_n\setminus\Lambda}:\Omega\to E^{\Delta_n\setminus\Lambda}$;
  we are only interested in $n$ so large that $\Delta_n\supset\Delta$.
  For such $n$, we claim that
  \[
    \mu_n^\omega\in\mathcal C(\mathcal B_{\Lambda,\Delta,\omega}(\gamma))
  \]
  almost surely (by which we mean: for $\mu$-almost every $\omega$).
  Equation~\ref{lemma_sfe_inf_sup} of Lemma~\ref{lemma_sfe}
  implies that
  \[
    \inf_{\rho\in\mathcal B_{\Delta_n}(\gamma)}
    \mathcal H(\mu_{\Delta_n}|\rho)
    =
    0.
  \]
  This implies that
  \[
    \inf_{\rho\in\mathcal B_{\Delta_n}(\gamma)}
    \left(
      \mathcal H(\mu_{\Delta_n\setminus\Lambda}|\rho_{\Delta_n\setminus\Lambda})
      +
      \int_{E^{\Delta_n\setminus\Lambda}} \mathcal H(\mu_n^\omega|\rho^\omega)d\mu_{\Delta_n\setminus\Lambda}(\omega)
      \right)
      =
      0,
  \]
  where $\rho^\omega$ is the r.c.p.d.\ of $\rho$
  on $(E^\Lambda,\mathcal E^\Lambda)$
  corresponding to the projection map $E^{\Delta_n}\to E^{\Delta_n\setminus\Lambda}$.
  Remark that $\rho^\omega\in \mathcal B_{\Lambda,\Delta_n,\omega}(\gamma)$
  almost surely because $\rho\in\mathcal B_{\Delta_n}(\gamma)$ and by consistency of $\gamma$.
  This means that
  \[
    \inf_{\rho^\omega\in \mathcal B_{\Lambda,\Delta_n,\omega}(\gamma)}
    \mathcal H(\mu_n^\omega|\rho^\omega)=0,
  \]
  and therefore $\mu_n^\omega\in\mathcal C(\mathcal B_{\Lambda,\Delta_n,\omega}(\gamma))$,
  almost surely.
  But $\mathcal C(\mathcal B_{\Lambda,\Delta_n,\omega}(\gamma))\subset\mathcal C(\mathcal B_{\Lambda,\Delta,\omega}(\gamma))$
  because $\Delta\subset\Delta_n$,
  which proves the claim.

  For any $A\in\mathcal E^\Lambda$,
  the bounded martingale convergence theorem
  says that almost surely
  \[
    \mu_n^\omega(A)\to\mu_\Lambda^\omega(A).
  \]
  The set $\mathcal C(\mathcal B_{\Lambda,\Delta,\omega}(\gamma))$ is compact in the strong topology,
  hence
  $\mu_n^\omega\to\mu_\Lambda^\omega\in\mathcal C(\mathcal B_{\Lambda,\Delta,\omega}(\gamma))$
  almost surely.
\end{proof}

\begin{corollary}
  \label{corollary_second_half}
  If $\gamma$ is a weakly dependent specification and $\mu\in h_0(\gamma)$ satisfies $\mu(\Omega_\gamma)=1$,
  then $\mu$ is almost Gibbs.
\end{corollary}

\begin{proof}
  By the previous lemma,
  $\mu_\Lambda^\omega\in \mathcal B_{\Lambda,\omega}(\gamma)=\{\hat\gamma_\Lambda(\cdot,\omega)\}$
  for $\mu$-a.e.\ $\omega$, proving that $\mu$ is a DLR state.
\end{proof}


\begin{corollary}
  \label{lemma_applications_finite_energy}
  Let $\gamma$ denote a weakly dependent specification,
  and fix a measure $\lambda\in\mathcal B_{\{0\}}(\gamma)$.
  We pretend that $\lambda$ is a probability measure
  on the state space $(E,\mathcal E)$.
  Then there exists a constant $\varepsilon>0$
  such that,
  for any minimizer $\mu\in h_0(\gamma)$ and for any $\Lambda\in\mathcal S$,
  we have
  $
    \mu_{\Lambda}^\omega\geq(\varepsilon\lambda)^{\Lambda}
  $
  for $\mu$-almost every $\omega$.
  In other words, $\mu$ has finite energy.
\end{corollary}

In particular, if $E$ is finite
and every state $e\in E$ has positive probability with respect to $\lambda$,
then one may replace $\lambda$ by the counting measure on $E$,
which possibly has the effect of forcing us to take $\varepsilon$ smaller.
By doing so, we obtain the original finite energy formulation
of Burton and Keane~\cite{BURTON}.

\begin{proof}[Proof of Corollary~\ref{lemma_applications_finite_energy}]
Consider a weakly dependent specification $\gamma$,
and fix a probability measure $\lambda\in\mathcal B_{\{0\}}(\gamma)$.
The definition of a weakly dependent specification
and Lemma~\ref{lemma_mazur}
imply that
$\operatorname{Diam}^\infty\mathcal C(\mathcal B_{\{0\}}(\gamma))$
is finite,
and therefore there exists an $\varepsilon>0$
such that
$\mu\geq\varepsilon\lambda$
 for any $\mu\in\mathcal C(\mathcal B_{\{0\}}(\gamma))$.
 (In fact, it is easy to see that the choice $\varepsilon:=\exp-\operatorname{Diam}^\infty\mathcal C(\mathcal B_{\{0\}}(\gamma))$ suffices for this purpose.)

Claim that $\mu\geq (\varepsilon\lambda)^{\Lambda}$
for any $\mu\in\mathcal B_\Lambda(\gamma)$,
for fixed $\Lambda\in\mathcal S$.
Write $\mu=\nu\hat\gamma_\Lambda$
for some $\nu\in\mathcal P(\Omega,\mathcal F)$.
Without loss of generality, we suppose that $\nu=\nu\gamma_{\Lambda}$,
so that $\mu=\nu_\Lambda$.
We also have $\nu=\nu\prod_{x\in\Lambda}\gamma_{\{x\}}$.
By induction,
\[
\nu=\nu\prod\nolimits_{x\in\Lambda}\gamma_{\{x\}}\geq (\varepsilon\lambda)^\Lambda\times\nu_{S\setminus\Lambda}.
\]
This proves the claim.
The claim also proves that $\mu\geq (\varepsilon\lambda)^{\Lambda}$
for any $\mu\in\mathcal C(\mathcal B_\Lambda(\gamma))$,
which implies the corollary due to
Lemma~\ref{lemma_rcpd}.
\end{proof}

\subsection{Duality between random fields and specifications}

\begin{lemma}
  \label{lemma_duality_same_SFE}
  Let $\gamma$ denote a weakly dependent specification
  and $\nu$ a minimizer of $\gamma$.
  Then for any shift-invariant random field $\mu$, we have
  \[
    h(\mu|\gamma)
    =
    h(\mu|\nu)
    :=
    \lim_{n\to\infty}
    |\Delta_n|^{-1}
    \mathcal H_{\Delta_n}(\mu|\nu).
  \]
\end{lemma}

\begin{proof}
  We observe that
  $
  |\mathcal H_{\Delta_n}(\mu|\nu)-\mathcal H_{\Delta_n}(\mu|\nu\gamma_{\Delta_n})|\leq \operatorname{Diam}^\infty
  \mathcal C(\mathcal B_{\Delta_n}(\gamma))=o(|\Delta_n|)
  $
  as $n\to\infty$.
\end{proof}

Let us now investigate the relation between $\mathbb S$
and $\mathbb F$.
Define the relation $\sim$
on $\mathbb F$ by declaring that $\mu\sim\nu$
whenever $\mu\in h_0(\nu)$.

\begin{lemma}
  The relation $\sim$ is an equivalence relation on $\mathbb F$
  with $h_0(\mu)$ the equivalence class of $\mu\in\mathbb F$.
\end{lemma}

\begin{proof}
  Fix $\nu\in \mathbb F$.
  Clearly $\nu\sim\nu$, because $h(\nu|\nu)=0$.
  It suffices to show that $h_0(\mu)=h_0(\nu)$
  whenever $\mu\sim\nu$.
  Suppose that $\mu\sim\nu$.
  As $\nu\in\mathbb F$,
  there exists a specification $\gamma\in\mathbb S$
  such that $\nu\in h_0(\gamma)$.
  The previous lemma implies that
  $h_0(\nu)=h_0(\gamma)$,
  that is, $\mu\in h_0(\gamma)$,
  and therefore also $h_0(\mu)=h_0(\gamma)$.
  This proves that $h_0(\mu)=h_0(\nu)$.
\end{proof}

This is sufficient for the conclusions that were drawn in Subsection~\ref{subsection_main_results_abstract_rela}.

%% file: new/06_00_applications.tex

\section{Applications} \label{s_applications}

Most of the classical results on the variational principle follows directly from our new setting. In this section we will give several examples of this fact.
We derive new results for the loop $O(n)$ model and for the Ising model in a random percolation environment,
which is also called the Griffiths singularity random field.

%% file: new/06_01_absolutely_summable.tex
\subsection{Models with an absolutely summable interaction potential}
\label{subsection_abs_summable}
In this subsection we show how to derive naturally from our work the variational principle for absolutely summable potential as described in~\cite{GEORGII} or~\cite{RASSOUL}.
The model of interest is described
by a reference measure and a shift-invariant absolutely summable potential.
Write $\lambda$ for the \emph{reference measure},
which is a probability measure on the state space $(E,\mathcal E)$.
This measure informs us of the most random distribution of the state of an isolated vertex in the absence of any interaction.
Write $\Phi=(\Phi_A)_{A\in\mathcal S}$
for the interaction potential.
The potential encodes the interactions that exist
between the states at different sites.
Formally, an \emph{interaction potential} $\Phi=(\Phi_A)_{A\in\mathcal S}$
is a family of functions such that $\Phi_A:\Omega\to \mathbb R\cup\{\infty\}$
is $\mathcal F_A$-measurable.
The potential $\Phi$ is called \emph{shift-invariant} if
$\Phi_{\theta A}(\omega)=\Phi_A(\theta\omega)$
for any $A\in\mathcal S$, $\theta\in\Theta$, $\omega\in\Omega$.
The potential $\Phi$ is called \emph{absolutely summable}
if
\[
  \|\Phi\|
  :=
  \sum_{A\in\mathcal S,\,0\in A}\|\Phi_A\|_\infty
  <\infty,
\]
where $\|\cdot\|_\infty$ denotes the supremum norm.
It is thus assumed that $\Phi$ is shift-invariant and absolutely summable.

The potential induces a Hamiltonian.
For $\Lambda\in \mathcal S$ and $\Delta\subset\mathbb Z^d$,
define
\[
  H_{\Lambda,\Delta}
  :=
  \sum_{A\in\mathcal S,\,A\cap\Lambda\neq\varnothing,\,A\subset\Delta}
  \Phi_A.
\]
In particular, the \emph{Hamiltonians} are
the functions of the form $H_\Lambda:=H_{\Lambda,S}$,
where $\Lambda\in\mathcal S$.
The reference measure $\lambda$ and the potential $\Phi$
generate a \emph{Gibbs specification} $\gamma=(\gamma_\Lambda)_{\Lambda\in\mathcal S}$
defined
by
\[
  \gamma_\Lambda(A,\omega)
  :=
  \frac{1}{Z_\Lambda^\omega}
  \int_{E^\Lambda}
  1_A(\zeta\omega_{S\setminus\Lambda})
  e^{-H_\Lambda(\zeta\omega_{S\setminus\Lambda})}
  d\lambda^\Lambda(\zeta)
\]
for any $\Lambda\in\mathcal S$, $\omega\in\Omega$,
and $A\in\mathcal F$,
where $Z_\Lambda^\omega$ is the normalizing constant
\[
\numberthis
\label{eq_Z_def}
Z_\Lambda^\omega:=
\int_{E^\Lambda}
e^{-H_\Lambda(\zeta\omega_{S\setminus\Lambda})}
d\lambda^\Lambda(\zeta).
\]

The Hamiltonian $H_\Lambda$ is always bounded by $|\Lambda|\cdot\|\Phi\|$.
Moreover,
for absolutely summable potentials,
the strength of the interaction decreases with the range.
We aim to show two things:
that the specification $\gamma$ is weakly dependent,
and that $\Omega_\gamma=\Omega$.
In that case, Corollary~\ref{corollary_first_half} and Corollary~\ref{corollary_second_half} prove the variational principle, where all almost Gibbs measures
are Gibbs measures.
For the analysis it is convenient to define,
for $\Lambda,\Delta\in\mathcal S$,
\[
  \varepsilon_{\Lambda,\Delta}:=\sum_{A\in\mathcal S,\,A\cap\Lambda\neq\varnothing,\,A\not\subset\Delta}
  \|\Phi_A\|_\infty.
\]
Compare this to the definition of $H_{\Lambda,\Delta}$ --- the construction implies the inequality
 $\|H_\Lambda-H_{\Lambda,\Delta}\|_\infty\leq \varepsilon_{\Lambda,\Delta}$.
 The constants $\varepsilon_{\Lambda,\Delta}$ contain
 precisely all the information
 that we need for proving weak dependence and that $\Omega_\gamma=\Omega$.
 To see this, we first prove the following lemma.

\begin{lemma}
For any $\omega\in\Omega$ and $\Lambda,\Delta\in\mathcal S$,
we have
  $\operatorname{Diam}^\infty\mathcal C(\mathcal B_{\Lambda,\Delta,\omega})
  \leq 4\varepsilon_{\Lambda,\Delta}$.
\end{lemma}

\begin{proof}
Fix $\omega',\omega''\in\Omega$
such that $\omega_\Delta=\omega_\Delta'=\omega_\Delta''$.
Choose $\zeta\in E^\Lambda$.
Then
$H_{\Lambda,\Delta}(\zeta\omega_{S\setminus\Lambda}')=H_{\Lambda,\Delta}(\zeta\omega_{S\setminus\Lambda}'')$,
and the triangular inequality implies that
\begin{align*}
&
|
H_\Lambda(\zeta\omega_{S\setminus\Lambda}')
-
H_\Lambda(\zeta\omega_{S\setminus\Lambda}'')
|
\\&\qquad
\leq
|H_\Lambda(\zeta\omega_{S\setminus\Lambda}')-H_{\Lambda,\Delta}(\zeta\omega_{S\setminus\Lambda}')|
+
|H_\Lambda(\zeta\omega_{S\setminus\Lambda}'')-H_{\Lambda,\Delta}(\zeta\omega_{S\setminus\Lambda}'')|
\leq
2\varepsilon_{\Lambda,\Delta}
.
\end{align*}
This inequality and (\ref{eq_Z_def}) --- the definition of $Z_\Lambda^\omega$ --- imply that
\[
  |\log Z_\Lambda^{\omega'}-\log Z_\Lambda^{\omega''}|\leq 2\varepsilon_{\Lambda,\Delta}.
\]
The definition of the specification implies that
$
\hat\gamma_\Lambda(\cdot,\omega)=
\frac{1}{Z_\Lambda^\omega}e^{-H_\Lambda(\cdot\omega_{S\setminus\Lambda})}\lambda^\Lambda
$,
and therefore we deduce from the inequalities in the previous two displays that
$
\mathcal H^\infty(\hat\gamma_\Lambda(\cdot,\omega'),\hat\gamma_\Lambda(\cdot,\omega''))
\leq
4\varepsilon_{\Lambda,\Delta}$.
Conclude that
\[
  \operatorname{Diam}^\infty\mathcal C(\mathcal B_{\Lambda,\Delta,\omega})
  =\operatorname{Diam}^\infty\mathcal B_{\Lambda,\Delta,\omega}
  =
  \sup_{\omega',\omega''\in\Omega,\,\omega_\Delta=\omega'_\Delta=\omega_\Delta''}
  \mathcal H^\infty(\hat\gamma_\Lambda(\cdot,\omega'),\hat\gamma_\Lambda(\cdot,\omega''))
  \leq
  4\varepsilon_{\Lambda,\Delta}.\qedhere
\]
\end{proof}

We now simply employ the bound provided by the lemma,
in order to arrive at the variational principle.
To deduce the variational principle with Gibbs measures,
we must prove that the specification $\gamma$ is weakly dependent,
and that $\Omega_\gamma=\Omega$.
By the lemma,
we know that
\begin{enumerate}
  \item $\operatorname{Diam}^\infty\mathcal B_{\Delta_n}(\gamma)\leq 4\varepsilon_{\Delta_n,\Delta_n}$,
  \item $\operatorname{Diam}^\infty\mathcal B_{\Lambda,\Delta_n,\omega}(\gamma)\leq 4\varepsilon_{\Lambda,\Delta_n}$
  for any $\omega\in\Omega$.
\end{enumerate}
To prove weak dependence,
it is therefore sufficient to show that $\varepsilon_{\Delta_n,\Delta_n}=o(|\Delta_n|)$
as $n\to\infty$.
Similarly, to prove that $\Omega_\gamma=\Omega$,
it is sufficient to show that $\varepsilon_{\Lambda,\Delta_n}\to 0$
as $n\to\infty$ for any $\Lambda\in\mathcal S$,
as this would imply that
\[
  \operatorname{Diam}^\infty\mathcal B_{\Lambda,\omega}(\gamma)
  \leq
  \inf_{n\in\mathbb N}
  \operatorname{Diam}^\infty\mathcal C(\mathcal B_{\Lambda,\Delta_n,\omega}(\gamma))
  =0.
\]
Start with the latter.
It is immediate from the definition of $\varepsilon_{\Lambda,\Delta_n}$
that
\[
  \varepsilon_{\Lambda,\Delta_n}
  \leq
  \sum_{x\in\Lambda}
  \varepsilon_{\{x\},\Delta_n}
  =
  \sum_{x\in\Lambda}
  \varepsilon_{\{0\},\Delta_n-x}\to 0
\]
as $n\to\infty$, because $|\Lambda|$ and $||\Phi||$ are both finite.
This proves that  $\Omega_\gamma=\Omega$.
For weak dependence, decompose
\begin{align*}
  \varepsilon_{\Delta_n,\Delta_n}
  &\leq
  \sum\nolimits_{x\in\Delta_n}
  \varepsilon_{\{x\},\Delta_n}
  =
  \sum\nolimits_{x\in\Delta_n}
  \varepsilon_{\{0\},\Delta_n-x}
  \\&
  =
  \sum\nolimits_{x\in\Delta_{n-\lfloor\log n\rfloor}}
  \varepsilon_{\{0\},\Delta_n-x}
  +
  \sum\nolimits_{x\in\Delta_n\setminus\Delta_{n-\lfloor\log n\rfloor}}
  \varepsilon_{\{0\},\Delta_n-x}
\\& \leq
|\Delta_{n-\lfloor\log n\rfloor}|\cdot \varepsilon_{\{0\},\Delta_{\lfloor\log n\rfloor}}
+|\Delta_n\setminus\Delta_{n-\lfloor\log n\rfloor}|\cdot||\Phi||
=o(|\Delta_n|)
\end{align*}
as $n\to\infty$.

%% file: new/06_02_RCM.tex
\subsection{The random-cluster model}
\label{subsection_rcm}
Let us introduce the random-cluster model.
Fix
an \emph{edge-weight} $p\in(0,1)$
and a \emph{cluster-weight} $q\in(0,\infty)$.
The idea of the random-cluster model
is to perform
 independent bond percolation (with parameter $p$) on (a subset of)
the square lattice $\mathbb Z^d$,
and subsequently weight each configuration by $q$
raised to the number of percolation clusters in the resulting random graph.
To cast the random-cluster model into the formalism of this paper,
we must first choose a suitable state space $(E,\mathcal E)$
for the vertices $x\in\mathbb Z^d$, which allows us to encode
for each edge if it is open or not.
There exists a natural way to do this:
with each vertex $x$ we associate the $d$
edges of the form
$\{x,x+e_i\}$ with $1\leq i\leq d$.
The state space that we choose
is
\[
  E=\{0,1\}^{\{1,\dots,d\}},
\]
where for $\omega_x\in E$ the $i$-th coordinate
is a $1$ if the edge $\{x,x+e_i\}$ is open
and $0$ if it is closed.
For $e\in E$ we define $|e|:=|\{1\leq i\leq d:e_i=1\}|$,
the number of open edges encoded in $e$.
If $\omega\in E^\Lambda$ for some $\Lambda\in\mathcal S$,
then write $\|\omega\|:=\sum_{x\in\Lambda}|\omega_x|$.
If $\omega\in\Omega$ and $\Lambda\in\mathcal S$,
then define
$
  C(\omega,\Lambda)$
  to be
  the number of open clusters of $\omega$ that intersect $\Lambda$ or contain a vertex adjacent to $\Lambda$.
It is important to observe that
\[
\numberthis
\label{equation_useful_boundary_bound}
|C(\omega,\Lambda)-C(\zeta,\Lambda)|\leq 2|\partial\Lambda|
\]
if $\omega_\Lambda=\zeta_\Lambda$,
where $\partial\Lambda$ denotes the \emph{edge boundary}
of $\Lambda$, that is, set of edges
of the square lattice with exactly one endpoint in $\Lambda$.
We now introduce the
 specification $\gamma=(\gamma_\Lambda)_{\Lambda\in\mathcal S}$
corresponding to the random-cluster model.
For any
$\omega\in\Omega$,
$\Lambda\in\mathcal S$,
and $\zeta\in E^\Lambda$,
we define the \emph{weight function}
 \[
 w(\zeta,\omega,\Lambda):=p^{\|\zeta\|}
 (1-p)^{d|\Lambda|-\|\zeta\|}
 q^{C(\zeta\omega_{S\setminus\Lambda},\Lambda)}.
 \]
The probability kernel $\hat\gamma_\Lambda$ corresponding to the random-cluster model
is now defined by
\[
  \hat\gamma_\Lambda(\zeta,\omega)
  :=
  \frac{1}{Z^\omega_\Lambda}
  w(\zeta,\omega,\Lambda),
\]
where $Z^\omega_\Lambda$ is a suitable normalization constant.
The
complete, nonrestricted probability
kernel $\gamma_\Lambda$
is given by $\gamma_\Lambda(\cdot,\omega)=\hat\gamma_\Lambda(\cdot,\omega)\times\delta_{\omega_{S\setminus\Lambda}}$.
Let us now prove that the resulting specification $\gamma=(\gamma_\Lambda)_{\Lambda\in\mathcal S}$ is weakly dependent.
From (\ref{equation_useful_boundary_bound}) and the definition of $w$
it is clear that
\[
\left|
\log \frac{w(\zeta,\omega,\Lambda)}{ w(\zeta,\omega',\Lambda)}
\right|
\leq 2|\partial\Lambda||\log q|
\]
for all possible $\zeta$, $\omega$, $\omega'$, and $\Lambda$.
As a direct consequence
\[
\left|
\log \frac{Z_\Lambda^\omega}{Z_\Lambda^{\omega'}}
\right|
\leq 2|\partial\Lambda||\log q|,
\quad
\text{and}
\quad
\left|
\log \frac{\hat\gamma_\Lambda(\zeta,\omega)}{\hat\gamma_\Lambda(\zeta,\omega')}
\right|
\leq 4|\partial\Lambda||\log q|.
\]
The right inequality implies that
\[
  \operatorname{Diam}^\infty\mathcal B_{\Delta_n}(\gamma)
  \leq 4|\partial\Delta_n||\log q|
  =o(|\Delta_n|)
\]
as $n\to\infty$, which proves that the specification $\gamma$ is weakly dependent.

The goal is to prove the variational principle,
which asserts the equivalence in Equation~\ref{eq_var_pr_equivalence}
for any shift-invariant random field $\mu$.
Weak dependence of $\gamma$ gives us access to the framework that is developed
in this paper.
In particular, we have the following three results:
\begin{enumerate}
  \item There exists at least one shift-invariant measure $\mu$ such that $h(\mu|\gamma)=0$,
  \item If $\mu$ is a shift-invariant DLR state, then $\mu\in h_0(\gamma)$,
  \item If $\mu\in h_0(\gamma)$ and $\mu(\Omega_\gamma)=1$,
  then $\mu$ is almost Gibbs.
\end{enumerate}
To arrive at the variational principle,
it is now sufficient to prove that $\mu(\Omega_\gamma)=1$
whenever $\mu\in h_0(\gamma)$.
%
%
%

Define
\[
  \Omega':=\left\{
    \omega\in\Omega:
    \parbox{25em}{if $\zeta\in\Omega$ is any other configuration
    that
    equals  $\omega$ up to finitely many edges,
    then $\zeta$ has at most one infinite component}
  \right\}.
\]
It follows from the well-known argument of Burton and Keane
\cite{BURTON}
that $\mu(\Omega')=1$ whenever $\mu$ is a shift-invariant random field
with finite energy.
Minimizers of the specific free energy have finite energy due to
Corollary~\ref{lemma_applications_finite_energy}.
Thus, in order to deduce the variational principle
for the random-cluster model,
it suffices to demonstrate that $\Omega'\subset\Omega_\gamma$.

Fix $\omega\in\Omega'$,
and claim that $\omega\in\Omega_\gamma$.
This is well-known for the random-cluster model, but perhaps not in the language of this article;
 we give a concise proof.
 Fix $\Lambda\in\mathcal S$.
 We make the stronger claim that
 \[
 \mathcal B_{\Lambda,\Delta,\omega}(\gamma)=\{\hat\gamma_\Lambda(\cdot,\omega)\}
 \]
 for $\Delta$ sufficiently large.
 In other words,
 we claim that for some appropriate choice of $\Delta$,
 the measure
$\hat\gamma_\Lambda(\cdot,\omega)$ is invariant
under changing $\omega$ on the complement of $\Delta$.
The point is that the dependence of $\hat\gamma_\Lambda(\cdot,\omega)$
on $\omega$ is through the way that the percolation structure encoded in $\omega$
connects the vertices in the boundary of $\Lambda$ with paths
through the complement of $\Lambda$.
Choose $\Delta\in\mathcal S$ such that
\begin{enumerate}
  \item $\Delta$ contains $\Lambda$,
  \item If $x$ is adjacent to $\Lambda$
  and part of a finite $\omega$-cluster,
  then $\Delta$ contains that entire finite $\omega$-cluster and all vertices adjacent to it,
  \item If $x$ and $y$ are adjacent to $\Lambda$
  and contained in the infinite $\omega$-cluster,
  then $\Delta$ contains an open path from $x$ to $y$ through the complement of $\Lambda$.
\end{enumerate}
The choice $\omega\in\Omega'$ guarantees that the open path from $x$ to $y$ through the
complement of $\Lambda$ exists.
The merit of this choice of $\Delta$ is of course that
\[
  C(\xi,\Lambda)=C(\xi',\Lambda),
\]
whenever $\xi,\xi'\in\Omega$ are chosen such that $\xi_\Delta=\xi_\Delta'$
and $\xi_{\Delta\setminus\Lambda}=\xi'_{\Delta\setminus\Lambda}=\omega_{\Delta\setminus\Lambda}$.
In particular, this implies that
\[
w(\zeta,\omega,\Lambda)=w(\zeta,\omega',\Lambda)
\]
for any $\zeta\in E^\Lambda$ and for any $\omega'\in\Omega$
such that $\omega_\Delta'=\omega_\Delta$.
Conclude that $\hat\gamma_\Lambda(\cdot,\omega')=\hat\gamma_\Lambda(\cdot,\omega)$
for such $\omega'\in\Omega$,
which implies the claim.

%% file: new/06_03_weight.tex
\subsection{The loop $O(n)$ model}
\label{subsection_loopOn}
The arguments for the variational principle
for the random-cluster model
work for any weakly dependent model
in which the long-range interaction is due to
weight on percolation clusters, level sets, paths, or other large geometrical
objects which arise from the local structure (for the random-cluster model
this was the cluster-weight $q$).
The variational principle holds true for all such models.
Consider, for example, the loop $O(n)$ model.
In this model, one draws
disjoint loops on the hexagonal lattice;
the probability of a certain configuration depends on
the number of loops and on the number of loop edges
in that configuration.
It is thus a two-parameter model, much like the random-cluster model.
See the work of Peled and Spinka \cite{PELED} for a detailed introduction.
The loop $O(n)$ model may be formalized as follows:
it is a model of random functions from the faces
of the hexagonal lattice to $E=\{0,1\}$.
The number of level sets of these functions corresponds to the number of loops in the loop $O(n)$ model,
and the number of edges on which the function is not constant corresponds to the number of edges that are contained in a loop.
Remark that in this case the Burton and Keane argument tells us
that there is at most one infinite level set on which the function equals $0$,
and at most one infinite level set on which the function equals $1$.
If both infinite level sets are present, then they are clearly distinguished by their type.

%% file: new/06_04_Griffiths.tex

\subsection{The Griffiths singularity random field}
\label{subsection_ising}
The Griffiths singularity random field was introduced by Van Enter, Maes, Schonmann,
and Shlosman~\cite{VanENTER}.
They study the model in relation to the phenomenon of so-called
\emph{Griffiths singularities}.
The model depends on two parameters:
the \emph{percolation parameter} $p\in (0,1)$,
and the \emph{inverse temperature} $\beta\in\mathbb R$;
both are fixed throughout the discussion.
We take $\beta\geq 0$ without loss of generality, which corresponds to the ferromagnetic setting.
To draw from the Griffiths singularity random field $K_{p,\beta}$,
one first samples independent site percolation
with parameter $p$;
then, on each percolation cluster, one samples
an independent Ising model with parameter $\beta$.
The Griffiths singularity random field is thus
an Ising model in a random environment.

First, we introduce some notation.
A natural choice for the state space is
$E=\{-1,0,1\}$.
The state $0$ indicates a closed vertex,
while the state $\pm1$ indicates an open vertex of that spin.
Write $\mathcal E$ for the powerset of $E$, a $\sigma$-algebra,
and $\mathcal E_0$ for the $\sigma$-algebra on $E$ generated by
the function $1_0$.
Let $\mathcal F^0$ denote the product $\sigma$-algebra $\mathcal E_0^S$.
If $\omega\in\Omega$
or $\omega\in E^\Lambda$
for some $\Lambda\subset S$, then write $\Pi(\omega)\subset\mathbb Z^d$
for the set of open vertices.
We consider each configuration $\omega\in\Omega$
to be a function from $\mathbb Z^d$ to $\{-1,0,1\}$,
and in that light we treat
$|\omega|$, $-\omega$, and $1_\Lambda$
as configurations in $\Omega$
for any
 $\omega\in\Omega$
or $\Lambda\subset \mathbb Z^d$.
There is a natural ordering $\leq$ on $\Omega$;
 write $\omega^1\leq\omega^2$
whenever $\omega_x^1\leq\omega_x^2$ for any $x\in\mathbb Z^d$.
If
$\mu_1,\mu_2\in\mathcal P(\Omega,\mathcal F)$,
then write $\mu_1\preceq\mu_2$
if $\mu_1$ is stochastically dominated by
$\mu_2$, that is,
if there exists a coupling between $\mu_1$
and $\mu_2$ such that $\omega^1\leq \omega^2$
almost surely.
Finally, the square lattice $\mathbb Z^d$ has naturally associated to it
an edge set; write $xy$ (juxtaposition) for an unordered pair of neighboring vertices $x,y\in\mathbb Z^d$
in this graph.
Write $\partial\Lambda$ for the edge boundary of
any set $\Lambda\in\mathcal S$, as in the analysis of the random-cluster model.

\subsubsection{The Ising model on a finite graph}
For finite sets $\Lambda\in\mathcal S$,
the \emph{Ising model} in $\Lambda$
is a probability measure
on $E^\Lambda$ defined by
\[
\alpha_\Lambda(\omega)
\propto\prod_{xy\subset \Lambda}e^{-\beta\omega_x\omega_y}
\]
if $\omega_x=\pm 1$ for every $x\in \Lambda$,
and $\alpha_\Lambda(\omega)=0$ otherwise.
The following key identity is a corollary of the definition:
\[
\numberthis
\label{eq_ising_decomp}
  \alpha_{\Lambda}(\omega)
  =  \frac{1}{Z}
\cdot
  f_{\Lambda,\Delta}(\omega)
  \cdot \alpha_{\Lambda\cap\Delta}(\omega_{\Lambda\cap\Delta})\cdot\alpha_{\Lambda\setminus\Delta}(\omega_{\Lambda\setminus\Delta})
\]
for any $\Lambda,\Delta\in\mathcal S$ and $\omega\in E^\Lambda$,
where
\[
f_{\Lambda,\Delta}(\omega):=\prod\nolimits_{xy\subset \Lambda,\,xy\in\partial\Delta}
e^{-\beta\omega_x\omega_y}
\qquad\text{and}\qquad
  Z=\int_{E^\Lambda} f_{\Lambda,\Delta} d(\alpha_{\Lambda\cap\Delta}\times\alpha_{\Lambda\setminus\Delta}).
\]
In particular,
if $\Lambda\in\mathcal S$
and $\Delta$ a connected component of $\Lambda$,
 then  (\ref{eq_ising_decomp}) implies that
 $\alpha_\Lambda=\alpha_\Delta\times\alpha_{\Lambda\setminus\Delta}$.

 If $\Lambda\in\mathcal S$
 and $\omega\in E^\Delta$
 for some $\Lambda\subset\Delta\subset S$,
 then we sometimes
 write $\alpha_\Lambda(\omega)$
 for $\alpha_\Lambda(\omega_\Lambda)$.

\subsubsection{The Ising model on an infinite graph}

The Ising model on infinite subgraphs of $\mathbb Z^d$
is introduced in terms of the associated specification,
which is denoted by $\kappa=(\kappa_\Lambda)_{\Lambda\in\mathcal S}$.
Consider arbitrary $\Lambda\in\mathcal S$
and $\omega\in\Omega$.
Informally, the measure $\kappa_\Lambda(\cdot,\omega)\in\mathcal P(\Omega,\mathcal F)$
is the Ising model in the graph $\Pi(\omega)\cap\Lambda$
--- the edges inherited from the square lattice --- subject to boundary conditions provided by the configuration $\omega$.
Formally, $\kappa_\Lambda(\cdot,\omega)$ is the unique random field such that
\[
  \kappa_\Lambda(\zeta,\omega)
  \propto
  \prod_{
    \text{$
      xy\subset \Lambda$
      or
      $xy\in\partial\Lambda$
      }
  }e^{-\beta \zeta_x\zeta_y}
\]
for any $\zeta\in\Omega$ such that $\zeta_{S\setminus\Lambda}=\omega_{S\setminus\Lambda}$
and
$\Pi(\zeta)=\Pi(\omega)$,
and $\kappa_\Lambda(\zeta,\omega)=0$
for all other $\zeta$.
Of course, the only edges $xy$ that contribute to
the product in the display are the ones that are also contained
in $\Pi(\zeta)=\Pi(\omega)$.
As per usual, we abbreviate $\hat\kappa_\Lambda(\cdot,\omega):=\sigma_\Lambda(\kappa_\Lambda(\cdot,\omega))$,
and we observe that  $\alpha_\Lambda=\hat\kappa_\Lambda(\cdot,1_\Lambda)$ in this notation.

The interest is however in the Ising model
in the entire
graph induced by $\Pi(\omega)$.
By monotonicity,
the sequence of random fields
$
  (\kappa_{\Delta_n}(\cdot,|\omega|))_{n\in\mathbb N}
$
is decreasing with respect to $\preceq$,
and therefore tends to a limit in the $\mathcal L$-topology as $n\to\infty$.
Write $\kappa^+(\cdot,\omega)$
for this limit, and similarly write $\kappa^-(\cdot,\omega)$
for the limit of the increasing sequence $(\kappa_{\Delta_n}(\cdot,-|\omega|))_{n\in\mathbb N}$.
Remark that both $\kappa^-(\cdot,\omega)$
and $\kappa^+(\cdot,\omega)$ depend on the percolation structure $\Pi(\omega)$
of $\omega$ only, and not on the spins of the open sites.
In other words, $\kappa^+$ and $\kappa^-$ are probability kernels from $(\Omega,\mathcal F^0)$
to $(\Omega,\mathcal F)$.
A monotonicity argument implies
that
$\kappa^-(\cdot,\omega)\preceq \kappa^+(\cdot,\omega)$.
If the two measures are distinct,
then it is said that the Ising model \emph{magnetizes} on $\Pi(\omega)$.
Write $M\subset\Omega$
for the collection of configurations $\omega$
such that
the Ising model magnetizes on $\Pi(\omega)$.
The set $M$ is measurable with respect to $\mathcal F^0$.
It is also measurable with respect to $\mathcal F_{S\setminus\Lambda}^0$,
for any $\Lambda\in\mathcal S$.
In other words, $M$ is tail measurable.
If $\zeta\in\Omega-M$,
then another monotonicity argument
implies that $\kappa^+(\cdot,\zeta)$
is the unique random field such
that almost surely $\Pi(\omega)=\Pi(\zeta)$
and which is invariant under each probability kernel $\kappa_\Lambda$.
We finally state an important proposition,
which also follows from monotonicity.

\begin{proposition}
  The map $\omega\mapsto\kappa^+(\cdot,\omega)$
  is continuous --- both sides endowed with the $\mathcal L$-topology ---
  at some $\zeta\in\Omega$ if and only if $\zeta\not\in M$.
\end{proposition}

\subsubsection{The random percolation environment}
Write $P_p$ for the percolation measure with parameter $p$,
that is, the measure in which each vertex takes value
$1$ with probability $p$,
and value $0$ with probability $1-p$,
independently of all other vertices.
Note that we have a zero-one law for the tail-measurable event $M$ in $P_p$.
We therefore distinguish three phases at most:
one phase of subcritical percolation,
one phase of supercritical percolation but with $P_p(M)=0$,
and one phase of supercritical percolation with $P_p(M)=1$.
Clearly $P_p(M)=0$ in the subcritical percolation regime
as there are no infinite clusters almost surely
and therefore the infinite Ising model decomposes
into the product of infinitely many finite cluster Ising models.
The interesting regime is therefore the supercritical percolation regime.
Our goal is to prove the variational principle
for the nonmagnetic phase --- both in the subcritical and supercritical percolation regime.

\subsubsection{Below critical percolation}
Let us for now assume that we are in the subcritical
percolation regime $p<p_c$,
so that we avoid the presence of an infinite percolation cluster altogether.
The Griffiths singularity random field $K_{p,\beta}$
is simply defined by the equation
$K_{p,\beta}:=P_p\kappa^+$.
To sample from $K_{p,\beta}$, one first samples
the percolation structure $\zeta$ from $P_p$,
then one draws the final sample $\omega$
from the Ising model $\kappa^+(\cdot,\zeta)$,
which decomposes into a product of Ising models
on the finite clusters of $\Pi(\zeta)$
almost surely.

Fix $\Lambda\in\mathcal S$.
Observe that
$K_{p,\beta}$
is invariant under the kernel
which first resamples the percolation
structure
on $\Lambda$,
then resamples the Ising model
on each percolation cluster that intersects $\Lambda$.
This motivates the definition of a natural specification
associated to $K_{p,\beta}$.
First,
 consider those $\omega\in\Omega$
for which there is no infinite percolation cluster.
For any $\Lambda\in\mathcal S$,
write $\Gamma(\omega,\Lambda)\subset\mathbb Z^d$
for the union of $\omega$-open clusters that contain a vertex
that is in or adjacent to $\Lambda$.
Also write $\|\omega_\Lambda\|$
for the number of $\omega$-open vertices in $\Lambda$.
For such $\omega$ and $\Lambda$,
we define the probability measure $\hat\gamma_\Lambda(\cdot,\omega)$ by
\[
\numberthis
\label{eq_ising_kernel_def}
\hat\gamma_\Lambda(\zeta,\omega)
:=
\frac{1}{Z_\Lambda^\omega}
p^{\|\zeta\|}
(1-p)^{|\Lambda|-\|\zeta\|}
\alpha_{\Gamma(\zeta\omega_{S\setminus\Lambda},\Lambda)}(\zeta\omega_{S\setminus\Lambda}),
\]
where $Z_\Lambda^\omega$ is a suitable normalization constant,
and $\zeta$ ranges over $E^\Lambda$.
As per usual, the full kernel $\gamma_\Lambda$
is recovered through the equation $\gamma_\Lambda(\cdot,\omega)=\hat\gamma_\Lambda(\cdot,\omega)\times\delta_{\omega_{S\setminus\Lambda}}$.
It follows from this definition and the intuitive picture that
$K_{p,\beta}=K_{p,\beta}\gamma_\Lambda$ for every $\Lambda\in\mathcal S$,
even though we have not yet defined $\gamma_\Lambda(\cdot,\omega)$
for those $\omega$ with an infinite percolation cluster.

Let us now rewrite the previous definition of $\hat\gamma_\Lambda(\cdot,\omega)$
into an expression that is less intuitive but more useful for the analysis.
First, write $\xi:=\zeta\omega_{S\setminus\Lambda}$
and $\Gamma:=\Gamma(\xi,\Lambda)$.
Use (\ref{eq_ising_decomp}) to obtain
\begin{align*}
  \alpha_\Gamma(\xi)
  =
  \frac{
  f_{\Gamma,\Lambda}(\xi)
  \cdot
  \alpha_{\Gamma\cap\Lambda}(\zeta)
  \cdot
  \alpha_{\Gamma\setminus\Lambda}(\omega)
  }{
  (\alpha_{\Gamma\cap\Lambda}\times\alpha_{\Gamma\setminus\Lambda})(f_{\Gamma,\Lambda})
  }
\end{align*}
Note that $\Gamma\cap\Lambda=\Pi(\zeta)$.
The set $\Gamma(\zeta\omega_{S\setminus\Lambda},\Lambda)-\Lambda$
depends on $\omega_{S\setminus\Lambda}$ only,
and therefore $\alpha_{\Gamma\setminus\Lambda}(\omega)$ is independent
of $\zeta$.
We may therefore combine $\alpha_{\Gamma\setminus\Lambda}(\omega)$ with the normalization constant
in (\ref{eq_ising_kernel_def})
to obtain
\[
\hat\gamma_\Lambda(\zeta,\omega)
=
\frac{1}{Z_\Lambda^\omega}
p^{\|\zeta\|}
(1-p)^{|\Lambda|-\|\zeta\|}
\alpha_{\Pi(\zeta)}(\zeta)
\frac{
f_{\Gamma,\Lambda}(\xi)
}{
  (\alpha_{\Pi(\zeta)}\times\alpha_{\Gamma\setminus\Lambda})(f_{\Gamma,\Lambda})
};
\]
now with a different normalization constant.
If we write $f_\Lambda$
for the function
\[
  f_\Lambda(\omega):=\prod_{xy\in\partial\Lambda}
  e^{-\beta\omega_x\omega_y},
\]
then the previous equation simplifies to
\[
\numberthis
\label{eq_grising_spec_def_general_simplified}
\hat\gamma_\Lambda(\zeta,\omega)
=
\frac{1}{Z_\Lambda^\omega}
p^{\|\zeta\|}
(1-p)^{|\Lambda|-\|\zeta\|}
\alpha_{\Pi(\zeta)}(\zeta)
\frac{
f_{\Lambda}(\xi)
}{
  (
  \hat\kappa_\Lambda(\cdot,1_{\Pi(\zeta)})
  \times
  \sigma_{S\setminus\Lambda}
    (
      \kappa^+(\cdot,1_{\Pi(\omega)-\Lambda})
    )
  )
  (
    f_{\Lambda}
  )
}.
\]
This probability kernel is well-defined
for any $\omega$,
even if $\omega$ has infinite clusters or if the Ising model
magnetizes on $\Pi(\omega)$.
We shall take (\ref{eq_grising_spec_def_general_simplified})
as a definition for
each kernel $\gamma_\Lambda$.
The family $\gamma=(\gamma_\Lambda)_{\Lambda\in\mathcal S}$
so produced
is a specification.
The long-range interaction derives exclusively
from the appearance of the measure
$\kappa^+(\cdot,1_{\Pi(\omega)-\Lambda})$
in the
denominator in the fraction on the right in (\ref{eq_grising_spec_def_general_simplified}).
Recall that $M$ is tail measurable:
 the Ising model magnetizes
on $\Pi(\omega)$ if and only if the Ising model
magnetizes on $\Pi(\omega)-\Lambda$.
This leads to the following crucial observation.

\begin{proposition}
  Consider $\zeta\in\Omega$.
  If $\zeta\not\in M$,
  then the
   map $\omega\mapsto\gamma_\Lambda(\cdot,\omega)$
  is continuous --- both sides endowed with the $\mathcal L$-topology ---
  at $\zeta$ for any $\Lambda\in\mathcal S$.
In other words,  $\Omega_\gamma$ contains $\Omega-M$.
\end{proposition}

We claim that the specification $\gamma$ is weakly dependent.
The reasoning is similar to the discussion of the random-cluster model.
The dependence on $\omega$ in (\ref{eq_grising_spec_def_general_simplified})
is only through its appearance in the fraction on the right,
and its effect on the normalization constant
 $Z_\Lambda^\omega$.
 But the definition of $f_\Lambda$ implies that $|\log f_\Lambda|\leq |\partial\Lambda||\beta|$.
 The logarithm of the fraction in (\ref{eq_grising_spec_def_general_simplified})
 is therefore bounded by $2|\partial\Lambda||\beta|$.
 Much like for the random-cluster model,
 this implies that
 \[
 \left|
 \log \frac{Z_\Lambda^\omega}{Z_\Lambda^{\omega'}}
 \right|
 \leq 4|\partial\Lambda||\beta|
 \quad
 \text{and}
 \quad
 \left|
 \log \frac{\hat\gamma_\Lambda(\zeta,\omega)}{\hat\gamma_\Lambda(\zeta,\omega')}
 \right|
 \leq 8|\partial\Lambda||\beta|,
 \]
 and we conclude with the asymptotic bound
 \[
  \operatorname{Diam}^\infty \mathcal B_{\Delta_n}(\gamma)
  \leq 8 |\partial\Delta_n||\beta|=o(|\Delta_n|)
 \]
 as $n\to\infty$; the specification $\gamma$ is weakly dependent.
Note that the argument for weak dependence of $\gamma$
works for any choice of parameters $p\in(0,1)$
and $\beta\geq 0$, regardless of the phase that we work in.

\subsubsection{Below magnetization}
For the remainder of the theory, it is no longer
necessary to require $p<p_c$.
Instead, we fix the percolation parameter $p$ and inverse temperature $\beta$ subject only to
 $P_p(M)=0$.
Of course, the Griffiths singularity random field $K_{p,\beta}$ is defined by the equation $K_{p,\beta}:=P_p\kappa^+=P_p\kappa^-$.
This measure is a DLR state of the specification $\gamma$
as defined in (\ref{eq_grising_spec_def_general_simplified}).
Moreover, we observe that
$K_{p,\beta}(M)=P_p(M)=0$,
and therefore $K_{p,\beta}$ is supported on $\Omega_\gamma$.
In other words, $K_{p,\beta}$ is almost Gibbs with respect
to $\gamma$.
Our final goal is to prove the following theorem.

\begin{theorem}
  \label{thm_GRISING_GOAL}
  If the parameters $p$ and $\beta$
  are such that $P_p(M)=0$,
  then $h_0(\gamma)=\{K_{p,\beta}\}$.
\end{theorem}

This statement is stronger than the variational principle,
it also implies that $K_{p,\beta}$ is the unique DLR
state of $\gamma$, and that $K_{p,\beta}$ is the unique minimizer of $\gamma$.

\begin{proof}[Proof of Theorem~\ref{thm_GRISING_GOAL}]
%
%
%
Fix $\mu\in h_0(\gamma)$.
Then $\mu\in h_0(K_{p,\beta})$.
Remark that $K_{p,\beta}|_{\mathcal F^0}=P_p|_{\mathcal F^0}$;
sampling the Ising model on the percolation clusters
alters the spins on those clusters, but not the percolation structure itself.
Observe that
\begin{align*}
  h(\mu|K_{p,\beta})
  &=\lim_{n\to\infty}|\Delta_n|^{-1}\mathcal H_{\mathcal E^{\Delta_n}}(\mu_{\Delta_n}|\sigma_{\Delta_n}(K_{p,\beta}))
  \\&\geq
  \lim_{n\to\infty}|\Delta_n|^{-1}\mathcal H_{\mathcal E_0^{\Delta_n}}(\mu_{\Delta_n}|\sigma_{\Delta_n}(K_{p,\beta}))
  =
  \lim_{n\to\infty}|\Delta_n|^{-1}\mathcal H_{\mathcal E_0^{\Delta_n}}(\mu_{\Delta_n}|\sigma_{\Delta_n}(P_p)).
\end{align*}
What we read on the last line in this display
is exactly the SFE of $\mu|_{\mathcal F^0}$
with respect to $P_p|_{\mathcal F^0}$.
But $P_p|_{\mathcal F^0}$ is a Gibbs measure
with respect to an independent specification,
which has a unique minimizer.
We chose $\mu$ such that $h(\mu|K_{p,\beta})=0$,
which now implies that
$\mu|_{\mathcal F^0}=P_p|_{\mathcal F^0}$.
We observe in particular that
$\mu(M)=0$,
and consequently $\mu(\Omega_\gamma)=1$.
Therefore $\mu$ is almost Gibbs with respect to $\gamma$.
Finally,
we observe that $\gamma_\Lambda=\gamma_\Lambda\kappa_\Lambda$.
This implies
that $\mu$ is also a DLR state of the specification $\kappa$.
But the Ising model is nonmagnetizing on $\Pi(\omega)$ for $\mu$-almost every $\omega$,
and therefore $\mu$ is also invariant under the probability kernel $\kappa^+$.
This kernel is $\mathcal F^0$-measurable;
conclude that $\mu=(\mu|_{\mathcal F^0})\kappa^+=P_p\kappa^+=K_{p,\beta}$.
\end{proof}

%% file: new/99_acknowledgement.tex
\section*{Acknowledgement}
The authors are grateful to
Nilanjana Datta,
Aernout van Enter,
Geoffrey Grimmett,
James Norris, and
Peter Winkler
for many useful discussions.
The authors would like to express
their special gratitude to Nathana\"el
Berestycki for enabling them to collaborate on this paper.

The first author was supported by the Department of
Pure Mathematics and Mathematical Statistics, University of Cambridge, the UK
Engineering and Physical Sciences Research Council grant EP/L016516/1,
and the Shapiro Visitor Program of the Department of Mathematics, Dartmouth College.
The second author was supported by
the UK
Engineering and Physical Sciences Research Council grant EP/L018896/1.

%% file: ms.bbl
\providecommand{\bysame}{\leavevmode\hbox to3em{\hrulefill}\thinspace}
\providecommand{\MR}{\relax\ifhmode\unskip\space\fi MR }
\providecommand{\MRhref}[2]{%
  \href{http://www.ams.org/mathscinet-getitem?mr=#1}{#2}
}
\providecommand{\href}[2]{#2}
\begin{thebibliography}{MRV99b}

\bibitem[BK88]{bricmont1988}
J.~Bricmont and A.~Kupiainen, \emph{Phase transition in the $3$d random field
  ising model}, Comm. Math. Phys. \textbf{116} (1988), no.~4, 539--572.

\bibitem[BK89]{BURTON}
R.~M. Burton and M.~Keane, \emph{Density and uniqueness in percolation}, Comm.
  Math. Phys. \textbf{121} (1989), no.~3, 501--505.

\bibitem[Dat09]{DATTA}
Nilanjana Datta, \emph{Min-and max-relative entropies and a new entanglement
  monotone}, IEEE Trans. Inform. Theory \textbf{55} (2009), no.~6, 2816--2826.

\bibitem[EFS93]{VanEnterFernandezSokal}
{Aernout C. D. van} Enter, Roberto Fern{\'a}ndez, and Alan~D. Sokal,
  \emph{Regularity properties and pathologies of position-space
  renormalization-group transformations: scope and limitations of {G}ibbsian
  theory}, J. Stat. Phys. \textbf{72} (1993), no.~5, 879--1167.

\bibitem[EMSS00]{VanENTER}
{Aernout van} Enter, Christian Maes, Roberto~H. Schonmann, and Senya Shlosman,
  \emph{The {G}riffiths singularity random field}, American Mathematical
  Society Translations (R.~A. Minlos, Senya Shlosman, and Yu.~M. Suhov, eds.),
  2, vol. 198, American Mathematical Society, 2000, pp.~51--58.

\bibitem[EV04]{VanEnterVerbitskiy}
{Aernout van} Enter and Evgeny Verbitskiy, \emph{On the variational principle
  for generalized {G}ibbs measures}, arXiv preprint math-ph/0410052v1 (2004).

\bibitem[Fer06]{FERNANDEZ}
Roberto Fern{\'a}ndez, \emph{{G}ibbsianness and non-{G}ibbsianness in lattice
  random fields}, Mathematical Statistical Physics: Lecture Notes of the Les
  Houches Summer School 2005 (Anton Bovier, Fran\c{c}ois Dunlop, {Aernout van}
  Enter, Frank~Den Hollander, and Jean Dalibard, eds.), Elsevier Science, 2006,
  pp.~731--798.

\bibitem[FLR03]{MR1964281}
Roberto Fern\'{a}ndez, Arnaud {Le Ny}, and Frank Redig, \emph{Variational
  principle and almost quasilocality for renormalized measures}, J. Statist.
  Phys. \textbf{111} (2003), no.~1-2, 465--478.

\bibitem[Geo11]{GEORGII}
Hans-Otto Georgii, \emph{Gibbs measures and phase transitions}, De Gruyter
  Studies in Mathematics, vol.~9, Walter de Gruyter, 2 ed., 2011.

\bibitem[Gri06]{GRIMMETT}
Geoffrey Grimmett, \emph{The random-cluster model}, Grundlehren der
  mathematischen Wissenschaften, vol. 333, Springer-Verlag, 2006.

\bibitem[HM02]{HAJLASZ}
Piotr Haj{\l}asz and Jan Mal{\'y}, \emph{Approximation in {S}obolev spaces of
  nonlinear expressions involving the gradient}, Ark. Mat. \textbf{40} (2002),
  no.~2, 245--274.

\bibitem[KLR04]{KULSKE}
Christof K{\"u}lske, Arnaud {Le Ny}, and Frank Redig, \emph{Relative entropy
  and variational properties of generalized {G}ibbsian measures}, Ann. Probab.
  \textbf{32} (2004), no.~2, 1691--1726.

\bibitem[Lef99]{MR1706769}
R.~Lefevere, \emph{Variational principle for some renormalized measures}, J.
  Statist. Phys. \textbf{96} (1999), no.~1-2, 109--133.

\bibitem[LPS95]{LPS}
{J. T.} Lewis, C.-E. Pfister, and W.~G. Sullivan, \emph{Entropy, concentration
  of probability and conditional limit theorems}, Markov Process. Relat.
  \textbf{1} (1995), no.~3, 319--386.

\bibitem[MRV99a]{MAES}
C.~Maes, F.~Redig, and A.~{Van Moffaert}, \emph{Almost {G}ibbsian versus weakly
  {G}ibbsian measures}, Stochastic Process. Appl. \textbf{79} (1999), no.~1,
  1--15.

\bibitem[MRV99b]{MR1706773}
\bysame, \emph{The restriction of the {I}sing model to a layer}, J. Statist.
  Phys. \textbf{96} (1999), no.~1-2, 69--107.

\bibitem[Pfi02]{Pfister2002}
Charles-Edouard Pfister, \emph{Thermodynamical aspects of classical lattice
  systems}, In and Out of Equilibrium: Probability with a Physics Flavor
  (Vladas Sidoravicius, ed.), Birkh{\"a}user Boston, 2002, pp.~393--472.

\bibitem[PS17]{PELED}
Ron Peled and Yinon Spinka, \emph{Lectures on the spin and loop {$O(n)$}
  models}, arXiv preprint arXiv:1708.00058v1 (2017).

\bibitem[PV95]{PFISTER}
Charles-Edouard Pfister and Koen~Vande Velde, \emph{Almost sure quasilocality
  in the random cluster model}, J. Stat. Phys. \textbf{79} (1995), no.~3,
  765--774.

\bibitem[RS15]{RASSOUL}
Firas {Rassoul-Agha} and Timo Sepp{\"a}l{\"a}inen, \emph{A course on large
  deviations with an introduction to {G}ibbs measures}, Graduate Studies in
  Mathematics, vol. 162, American Mathematical Society, 2015.

\bibitem[Sep93a]{MR1227034}
Timo Sepp\"{a}l\"{a}inen, \emph{Large deviations for lattice systems. {I}.
  {P}arametrized independent fields}, Probab. Theory Related Fields \textbf{96}
  (1993), no.~2, 241--260.

\bibitem[Sep93b]{MR1240718}
\bysame, \emph{Large deviations for lattice systems. {II}. {N}onstationary
  independent fields}, Probab. Theory Related Fields \textbf{97} (1993),
  no.~1-2, 103--112.

\bibitem[Sep95]{MR1344727}
\bysame, \emph{Entropy, limit theorems, and variational principles for
  disordered lattice systems}, Comm. Math. Phys. \textbf{171} (1995), no.~2,
  233--277.

\bibitem[Sep98]{SEPPALAINEN}
\bysame, \emph{Entropy for translation-invariant random-cluster measures}, Ann.
  Probab. \textbf{26} (1998), no.~3, 1139--1178.

\bibitem[SZ91]{Stroock}
{Daniel W. } Stroock and {Ofer} Zeitouni, \emph{Microcanonical distributions,
  {G}ibbs states, and the equivalence of ensembles}, Random Walks, Brownian
  Motion, and Interacting Particle Systems (Rick Durrett and Harry Kesten,
  eds.), Birkh{\"a}user Boston, 1991, pp.~399--424.

\bibitem[Ver10]{MR2791060}
Evgeny Verbitskiy, \emph{Variational principle for fuzzy {G}ibbs measures},
  Mosc. Math. J. \textbf{10} (2010), no.~4, 811--829.

\bibitem[Zeg91]{ZEGARLINSKI}
Boguslav Zegarlinski, \emph{Interactions and pressure functionals for
  disordered lattice systems}, Comm. Math. Phys. \textbf{139} (1991), no.~2,
  305--339.

\end{thebibliography}
